\documentclass[11pt]{amsart}
\usepackage{amsmath, amssymb, comment, graphicx, url, amsthm, multicol, latexsym, enumerate, bbm, mathtools, physics, microtype, cite, xcolor, float}
\usepackage{graphicx} 
\usepackage{units}
\usepackage[all,cmtip]{xy}
\usepackage{wrapfig}
\usepackage[hmargin=1in]{geometry} 
\usepackage{tikz}
\usepackage{tikzsymbols}
\usetikzlibrary{shapes.geometric, arrows}
\usepackage[cm]{sfmath}
\usepackage{hyperref}
\usepackage{indentfirst}

\usepackage{subcaption}

\definecolor{andresblue}{rgb}{0,0.72,0.92}
\definecolor{andrespink}{rgb}{1,0,1}
\definecolor{orange}{rgb}{1,0.5,0}

\tikzset{
    back/.style={
        loosely dotted,
        thin
    },
    edge/.style={
        color=black,
        thick
    },
    facet/.style={
        fill=andresblue,
        fill opacity=0.333333
    },
    vertex/.style={
        inner sep=1.5pt,
        circle,
        draw=black,
        fill=andrespink,
        thick
    },
    subvertex/.style={
        inner sep=1.5pt,
        circle,
        draw=orange!75!black,
        fill=orange,
        thick
    },
    subedge/.style={
        color=orange!95!black,
        thick
    },
    subfacet/.style={
        fill=orange!95!black,
        fill opacity=0.5
    },
    subback/.style={
        dotted,
        thick
    },
    slabel/.style n args={2}{
        label={[font=\scriptsize,black,#1]:{#2}}
    },
}

\hypersetup{
    colorlinks=true,
    linkcolor=blue,
    citecolor=magenta,
    filecolor=magenta,      
    urlcolor=magenta,
}

\newtheorem{theorem}{Theorem}[section]
\newtheorem{proposition}[theorem]{Proposition}

\newtheorem{lemma}[theorem]{Lemma}
\newtheorem{corollary}[theorem]{Corollary}

\theoremstyle{definition}

\newtheorem{definition}[theorem]{Definition}

\newtheorem{example}[theorem]{Example}

\theoremstyle{remark}

\newcommand{\defterm}[1]{\emph{#1}}

\newcommand{\R}{\mathbb{R}}

\newcommand{\Z}{\mathbb{Z}}

\renewcommand{\va}{\mathbf{a}}
\renewcommand{\vb}{\mathbf{b}}
\newcommand{\vc}{\mathbf{c}}
\newcommand{\ve}{\mathbf{e}}
\newcommand{\vn}{\mathbf{n}}
\newcommand{\vv}{\mathbf{v}}
\newcommand{\vx}{\mathbf{x}}
\newcommand{\vz}{\mathbf{z}}
\newcommand{\bv}[1]{\mathbf{#1}}
\newcommand{\pprime}{{\prime\prime}}
\newcommand{\qso}{\quad\text{so}\quad}
\newcommand{\set}[2]{\left\{#1\vphantom{#2}
\;\right\rvert\;\left.
#2\vphantom{#1}\right\}}
\renewcommand{\P}{\mathcal{P}}
\newcommand{\SEP}{\mathrm{SEP}}
\DeclareMathOperator{\rvol}{vol}
\DeclareMathOperator{\Span}{span}
\newcommand{\etal}{\textit{et al.} }
\newcommand{\Aut}{\textrm{Aut}}
\newcommand{\G}{\mathcal{G}}

\DeclareMathOperator{\conv}{conv}

\makeatletter 
\newtheorem*{rep@theorem}{\rep@title}\newcommand{\newreptheorem}[2]{%
\newenvironment{rep#1}[1]{%
\def\rep@title{\bf #2 \ref{##1}}%
\begin{rep@theorem}}%
{\end{rep@theorem}}}
\makeatother
\newreptheorem{theorem}{Theorem}

\newtheorem*{rep@proposition}{\rep@title}\newcommand{\newrepproposition}[2]{%
\newenvironment{rep#1}[1]{%
\def\rep@title{\bf #2 \ref{##1}}%
\begin{rep@proposition}}%
{\end{rep@proposition}}}
\makeatother
\newreptheorem{proposition}{Proposition}

\begin{document}

\title{Unitary Actions and Equivariant Volumes of\\ Symmetric Edge Polytopes}

\author{Tito Augusto Cuchilla}
\address{\scriptsize{Department of Mathematics, Harvey Mudd College}}
\email{\scriptsize{tcuchilla@g.hmc.edu}}

\author{Joseph Hound}
\address{\scriptsize{Department of Mathematics, Claremont Graduate University}}
\email{\scriptsize{joseph.hound@cgu.edu}}

\author{Cole Plepel}
\address{\scriptsize{Department of Mathematics, Harvey Mudd College}}
\email{\scriptsize{cplepel@g.hmc.edu}}

\author{Andr\'es R. Vindas-Mel\'endez}
\address{\scriptsize{Department of Mathematics, Harvey Mudd College}, \url{https://math.hmc.edu/arvm/}}
\email{\scriptsize{avindasmelendez@g.hmc.edu}}

\author{Louis Ye}
\address{\scriptsize{Department of Mathematics, Harvey Mudd College}}
\email{\scriptsize{yye@g.hmc.edu}}


\begin{abstract}
    The symmetric edge polytope ($\SEP$) of a finite simple graph $G$ is a centrally symmetric lattice polytope whose vertices are defined by the edges of the graph.
    Among the information encoded by these polytopes are the symmetries of the graph, which appear as symmetries of the polytope.
    We describe the rigid symmetries of these polytopes, and show that $\SEP$s are unitarily equivalent exactly when their associated graphs are isomorphic.
    We then find an explicit relationship between the relative volumes of the subsets of the symmetric edge polytope $\SEP(G)$ fixed by the natural action of symmetric group elements and the symmetric edge polytopes of smaller graphs to which the subsets are linearly equivalent.
    We also provide a vertex description of the fixed polytopes and find a description of the symmetric edge polytopes to which they are equivalent, in terms of contractions of the graph $G$ induced by the cycle decompositions of the permutations under which the subsets are fixed.
    Specializations of our results provide equivalence and volume relationships for fixed polytopes of symmetric edge polytopes of complete graphs (equivalently, for fixed polytopes of root polytopes of type $A_n$), and describe the symmetry group of this family of polytopes.
\end{abstract}


\maketitle

\vspace{-.5cm}
\section{Introduction}

The \defterm{symmetric edge polytope} associated with a simple connected graph $G=([n], E)$ is defined as
    \begin{equation*}
        \SEP(G)=\conv\set{\pm(\ve_i-\ve_j)\in \R^n}{(i,j)\in E},
    \end{equation*}
where $\ve_i$ denotes $i^{\text{th}}$ standard basis in $\R^n$.
These polytopes were first introduced in \cite{sep-paper} and have since been widely studied.
Symmetric edge polytopes are of interest in many areas, including polyhedral geometry, (equivariant) Ehrhart theory, applications to
algebraic Kuramoto equations, and matroid theory \cite{ChenDavis, ChenDavisMehta, DAliJuhnkeKohneVenturello, DAliDelucchi, HigashitaniJochemkoMichalek, KalmanTothmeresz, OhsugiTsuchiya, BraunBrueggeKahle, BraunBruegge, DAliJuhnkeKoch, ClarkeHigashitaniKolbl}.
See Figure \ref{fig:sep-examples} for images of the of symmetric edge polytopes for the complete graph $K_4$ and the star graph $\textrm{Star}_3$.

Symmetric edge polytopes capture edge information of graphs through the discrete geometry of a convex polytope.
In this paper, we study the symmetries, equivalences, and fixed polytopes induced by symmetric group actions on symmetric edge polytopes.
Namely, the symmetric group $S_n$ acts on $\SEP(G) \subset \R^n$ by permuting coordinates, that is, a permutation $\sigma \in S_n$ acts by $\sigma\cdot\ve_i=\ve_{\sigma(i)}$.

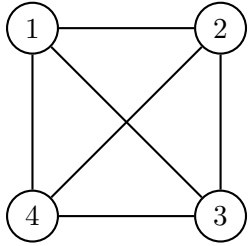
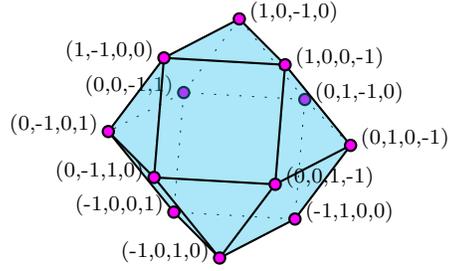
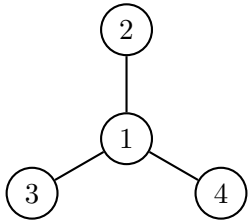
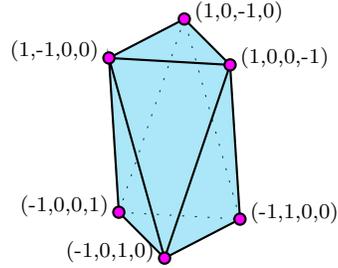
\begin{figure*}
    \centering
    \begin{subfigure}{0.5\textwidth}
        \centering
        \begin{tikzpicture}[
            scale=1.25,
            every node/.style={circle,thick,draw}
            ]
            \node (a) at (-1,1) {1};
            \node (b) at (1,1) {2};
            \node (c) at (1,-1) {3};
            \node (d) at (-1,-1) {4};
            \path [thick] (a) edge (b);
            \path [thick] (a) edge (c);
            \path [thick] (a) edge (d);
            \path [thick] (b) edge (c);
            \path [thick] (b) edge (d);
            \path [thick] (c) edge (d);
        \end{tikzpicture}
        \subcaption{$K_4$, the complete graph on 4 vertices}
    \end{subfigure}%
    \begin{subfigure}{0.5\textwidth}
        \centering
        \begin{tikzpicture}[
            x={(0.374024cm, -0.375379cm)},
            y={(0.927419cm, 0.151412cm)},
            z={(-0.000023cm, 0.914421cm)},
            scale=1.15,
            ]
            \draw[edge,back] (0.00000, 0.81650, 2.30940) -- (0.00000, 1.63299, 1.15470);
            \draw[edge,back] (0.00000, 0.81650, 2.30940) -- (-0.70711, 0.40825, 1.15470);
            \draw[edge,back] (1.41421, 1.63299, 1.15470) -- (0.00000, 1.63299, 1.15470);
            \draw[edge,back] (0.00000, 1.63299, 1.15470) -- (-0.70711, 0.40825, 1.15470);
            \draw[edge,back] (0.00000, 1.63299, 1.15470) -- (0.70711, 1.22474, 0.00000);
            \draw[edge,back] (-0.70711, 0.40825, 1.15470) -- (0.00000, -0.81650, 1.15470);
            \draw[edge,back] (-0.70711, 0.40825, 1.15470) -- (0.00000, 0.00000, 0.00000);
            \draw[edge,back] (0.70711, 1.22474, 0.00000) -- (0.00000, 0.00000, 0.00000);
            \node[vertex,slabel={right}{(0,1,-1,0)}] at (0.00000, 1.63299, 1.15470)     {};
            \node[vertex,slabel={left=0.0035cm}{(0,0,-1,1)}] at (-0.70711, 0.40825, 1.15470)     {};
            \fill[facet] (0.70711, -0.40825, 2.30940) -- (1.41421, 0.81650, 2.30940) -- (0.00000, 0.81650, 2.30940) -- cycle {};
            \fill[facet] (2.12132, 0.40825, 1.15470) -- (1.41421, 0.81650, 2.30940) -- (1.41421, 1.63299, 1.15470) -- cycle {};
            \fill[facet] (1.41421, -0.81650, 1.15470) -- (0.70711, -0.40825, 2.30940) -- (1.41421, 0.81650, 2.30940) -- (2.12132, 0.40825, 1.15470) -- cycle {};
            \fill[facet] (0.00000, -0.81650, 1.15470) -- (0.70711, -0.40825, 2.30940) -- (1.41421, -0.81650, 1.15470) -- cycle {};
            \fill[facet] (1.41421, 0.00000, 0.00000) -- (2.12132, 0.40825, 1.15470) -- (1.41421, 1.63299, 1.15470) -- (0.70711, 1.22474, 0.00000) -- cycle {};
            \fill[facet] (1.41421, 0.00000, 0.00000) -- (2.12132, 0.40825, 1.15470) -- (1.41421, -0.81650, 1.15470) -- cycle {};
            \fill[facet] (0.00000, 0.00000, 0.00000) -- (0.00000, -0.81650, 1.15470) -- (1.41421, -0.81650, 1.15470) -- (1.41421, 0.00000, 0.00000) -- cycle {};
            \draw[edge] (1.41421, 0.81650, 2.30940) -- (0.00000, 0.81650, 2.30940);
            \draw[edge] (1.41421, 0.81650, 2.30940) -- (0.70711, -0.40825, 2.30940);
            \draw[edge] (1.41421, 0.81650, 2.30940) -- (1.41421, 1.63299, 1.15470);
            \draw[edge] (1.41421, 0.81650, 2.30940) -- (2.12132, 0.40825, 1.15470);
            \draw[edge] (0.00000, 0.81650, 2.30940) -- (0.70711, -0.40825, 2.30940);
            \draw[edge] (0.70711, -0.40825, 2.30940) -- (1.41421, -0.81650, 1.15470);
            \draw[edge] (0.70711, -0.40825, 2.30940) -- (0.00000, -0.81650, 1.15470);
            \draw[edge] (1.41421, 1.63299, 1.15470) -- (2.12132, 0.40825, 1.15470);
            \draw[edge] (1.41421, 1.63299, 1.15470) -- (0.70711, 1.22474, 0.00000);
            \draw[edge] (2.12132, 0.40825, 1.15470) -- (1.41421, -0.81650, 1.15470);
            \draw[edge] (2.12132, 0.40825, 1.15470) -- (1.41421, 0.00000, 0.00000);
            \draw[edge] (1.41421, -0.81650, 1.15470) -- (0.00000, -0.81650, 1.15470);
            \draw[edge] (1.41421, -0.81650, 1.15470) -- (1.41421, 0.00000, 0.00000);
            \draw[edge] (0.00000, -0.81650, 1.15470) -- (0.00000, 0.00000, 0.00000);
            \draw[edge] (0.70711, 1.22474, 0.00000) -- (1.41421, 0.00000, 0.00000);
            \draw[edge] (1.41421, 0.00000, 0.00000) -- (0.00000, 0.00000, 0.00000);
            \node[vertex,slabel={right}{(1,0,0,-1)}] at (1.41421, 0.81650, 2.30940)     {};
            \node[vertex,slabel={right}{(1,0,-1,0)}] at (0.00000, 0.81650, 2.30940)     {};
            \node[vertex,slabel={left}{(1,-1,0,0)}] at (0.70711, -0.40825, 2.30940)     {};
            \node[vertex,slabel={right}{(0,1,0,-1)}] at (1.41421, 1.63299, 1.15470)     {};
            \node[vertex,slabel={right}{(0,0,1,-1)}] at (2.12132, 0.40825, 1.15470)     {};
            \node[vertex,slabel={left}{(0,-1,1,0)}] at (1.41421, -0.81650, 1.15470)     {};
            \node[vertex,slabel={left}{(0,-1,0,1)}] at (0.00000, -0.81650, 1.15470)     {};
            \node[vertex,slabel={right}{(-1,1,0,0)}] at (0.70711, 1.22474, 0.00000)     {};
            \node[vertex,slabel={left}{(-1,0,1,0)}] at (1.41421, 0.00000, 0.00000)     {};
            \node[vertex,slabel={left}{(-1,0,0,1)}] at (0.00000, 0.00000, 0.00000)     {};
        \end{tikzpicture}
        \caption{$\SEP(K_4)$, the symmetric edge polytope of $K_4$}
    \end{subfigure}
    \smallskip
    
    \begin{subfigure}{0.5\textwidth}
        \centering
        \begin{tikzpicture}[
            scale=1.45,
            every node/.style={circle,thick,draw}
            ]
            \node (a) at (0,0) {1};
            \node (b) at (0,1) {2};
            \node (c) at (-0.86603,-0.5) {3};
            \node (d) at (0.86603,-0.5) {4};
            \path [thick] (a) edge (b);
            \path [thick] (a) edge (c);
            \path [thick] (a) edge (d);
        \end{tikzpicture}
        \subcaption{$\textrm{Star}_3$, the star graph with 3 spokes}
    \end{subfigure}%
    \begin{subfigure}{0.5\textwidth}
        \centering
        \begin{tikzpicture}[
            x={(0.374024cm, -0.375379cm)},
            y={(0.927419cm, 0.151412cm)},
            z={(-0.000023cm, 0.914421cm)},
            scale=1.15
            ]
            \coordinate (ad) at (1.41421, 0.81650, 2.30940);
            \coordinate (ac) at (0.00000, 0.81650, 2.30940);
            \coordinate (ab) at (0.70711, -0.40825, 2.30940);
            \coordinate (bd) at (1.41421, 1.63299, 1.15470);
            \coordinate (cd) at (2.12132, 0.40825, 1.15470);
            \coordinate (cb) at (1.41421, -0.81650, 1.15470);
            \coordinate (db) at (0.00000, -0.81650, 1.15470);
            \coordinate (ba) at (0.70711, 1.22474, 0.00000);
            \coordinate (ca) at (1.41421, 0.00000, 0.00000);
            \coordinate (da) at (0.00000, 0.00000, 0.00000);
            \draw[edge, back] (da) -- (ba);
            \draw[edge, back] (da) -- (ac);
            \draw[edge, back] (ba) -- (ac);
            \fill[facet] (da) -- (ab) -- (ca) -- cycle;
            \fill[facet] (ab) -- (ca) -- (ad) -- cycle;
            \fill[facet] (ca) -- (ad) -- (ba) -- cycle;
            \fill[facet] (ab) -- (ac) -- (ad) -- cycle;
            \draw[edge] (ac) -- (ab) -- (ad) -- (ac);
            \draw[edge] (da) -- (ab) -- (ca) -- (ad) -- (ba);
            \draw[edge] (da) -- (ca) -- (ba);
            \node[vertex,slabel={right}{(1,0,-1,0)}] at (ac) {};
            \node[vertex,slabel={right}{(1,0,0,-1)}] at (ad) {};
            \node[vertex,slabel={left}{(1,-1,0,0)}] at (ab) {};
            \node[vertex,slabel={left}{(-1,0,0,1)}] at (da) {};
            \node[vertex,slabel={left}{(-1,0,1,0)}] at (ca) {};
            \node[vertex,slabel={right}{(-1,1,0,0)}] at (ba) {};
        \end{tikzpicture}
        \caption{$\SEP(\textrm{Star}_3)$, the symmetric edge polytope of $\textrm{Star}_3$}
    \end{subfigure}
    \caption{Some graphs and their corresponding symmetric edge polytopes}
    \label{fig:sep-examples}
\end{figure*}

Letting $\SEP(G)^\sigma$ be the portion of $\SEP(G)$ fixed by the action of a permutation $\sigma$, our main result on fixed polytopes relates the volumes of fixed polytopes of symmetric edge polytopes to the volumes of the symmetric edge polytopes to which they are linearly equivalent.

\begin{reptheorem}{thm:sep-sub-vol}
    Let $G=([n],E)$ be a graph on $n$ vertices, let $\sigma\in S_n$ have cycle decomposition $\sigma=\sigma_1\cdots\sigma_m$, and let $G'=(\sigma,E')$ be the contraction of $G$ induced by $\sigma$ as a partition of $[n]$.
    Then,
    \begin{equation*}
    \rvol(\SEP(G)^\sigma)=\frac{\gcd(|\sigma_1|,\dots,|\sigma_m|)}{\prod_{i=1}^m|\sigma_i|}\rvol(\SEP(G')).
    \end{equation*}
\end{reptheorem}

The rest of the paper is as follows.
In Section \ref{sec:preliminaries}, we provide a brief introduction to concepts from discrete geometry, graph theory, and algebra used throughout this paper.
Then, in Section \ref{sec:invariant_actions}, we identify linear transformations under which symmetric edge polytopes are invariant, in terms of the underlying graph, and show that symmetric edge polytopes are unitarily equivalent if and only if their associated graphs are isomorphic.

In Section \ref{sec:fixed_polytopes}, we study the subsets of symmetric edge polytopes fixed by coordinate permutations, which have natural interpretations in terms of the underlying graphs.
In particular, we show that these fixed polytopes are linearly equivalent to the symmetric edge polytopes associated with contractions of the original graph, and we find an explicit relationship between the volume of these fixed polytopes and the symmetric edge polytopes to which they are linearly equivalent.

Finally, in Section \ref{sec:fixed_root_polytopes}, we specialize these results to a particular family of symmetric edge polytopes, which have been studied from various perspectives in diverse branches of mathematics.
Specifically, we identify the consequences of the equivalence between symmetric edge polytopes of complete graphs and fixed subpolytopes thereof and present an explicit formula for the volumes of these fixed polytopes.

\section{Preliminaries} \label{sec:preliminaries}

By $[n]$, we denote the set of integers from $1$ to $n$ and we denote the symmetric group on $[n]$ by $S_n$.
By abuse of notation, we treat a $k$-cycle $\sigma\in S_n$ as a $k$-element set containing the non-fixed elements of $[n]$.
Similarly, when convenient, we treat $\sigma\in S_n$ as a partition of $[n]$ consisting of the sets associated to the components of its cycle decomposition.
The complement of a set $S$ is denoted $S^C$.

When a group $\G$ acts on a set $S$, we use $g\cdot s$ to represent the result of applying the action of a group element $g\in \G$ to an element $s\in S$.
We will frequently deal with the action of the symmetric group, $S_n$, on the vector space $\R^n$ and polytopes in that space.
Unless otherwise specified, we will assume the natural coordinate permutation representation of $S_n$ as its action on $\R^n$, and extend this to polytopes via pointwise transformation.

\begin{definition}\label{def:e}
    For $i\in[n]$, we define $\ve_i\in\R^n$ to be the \defterm{standard basis vector} whose $i^{\text{th}}$ coordinate is $1$, with all other coordinates $0$.
    The dimension of this vector will be clear from context.
    For $S\subseteq[n]$, we define $\ve_S=\sum_{i\in S}\ve_i$, and we use $\bv1=\ve_{[n]}$ to denote the all ones vector.
\end{definition}

Our main results describe the volumes of fixed polytopes.

\begin{definition}
    Given a group $\G$, a polytope $\P$ is \defterm{$\G$-invariant} if $g\cdot\P=\P$ for all $g\in \G$, where $g\cdot\P=\set{g\cdot\vx}{\vx\in\P}$.

For a $\G$-invariant polytope $\P$ in $\R^n$ and a group element $g\in \G$ identified with its representation as a linear transformation acting on $\R^n$, we define the \defterm{fixed polytope} of $\P$ fixed by $g$ as
    \begin{equation*}
        \P^g = \set{\vx\in\P}{g\cdot\vx=\vx}.
    \end{equation*}
\end{definition}

In particular, we study the relative volume of these polytopes with respect to the integer lattice of their affine span, which we define presently.

\begin{definition}\label{def:rvol}
    We define the \defterm{relative volume} of a polytope $\P$, denoted $\rvol\P$, to be the ratio of its Euclidean volume to the Euclidean volume of the smallest parallelepiped whose affine span matches that of $\P$.
    This quantity is also known to be the leading coefficient in the Ehrhart quasipolynomial of $\P$ \cite{ccd}.
\end{definition}

When discussing relative volumes, it is also useful to be able to refer to families of polytopes whose relative volume is easily computed.
We define one such family here.

\begin{definition}\label{def:half-open-parallelepiped}
    The \defterm{half-open parallelepiped} formed by a set of vectors $\vv_1,\dots,\vv_m$ is defined as
    \begin{equation*}
        \set{\lambda_1\vv_1+\dots+\lambda_m\vv_m}{0\le\lambda_i<1\ \text{and}\ \sum_{i=1}^m\lambda_i=1}.
    \end{equation*}
\end{definition}

\begin{figure*}
    \begin{subfigure}{0.25\textwidth}
        \centering
        \begin{tikzpicture}[
            x={(0.374024cm, -0.375379cm)},
            y={(0.927419cm, 0.151412cm)},
            z={(-0.000023cm, 0.914421cm)},
            scale=1.0,
            ]
            \draw[edge,back] (0.00000, 0.81650, 2.30940) -- (0.00000, 1.63299, 1.15470);
            \draw[edge,back] (0.00000, 0.81650, 2.30940) -- (-0.70711, 0.40825, 1.15470);
            \draw[edge,back] (1.41421, 1.63299, 1.15470) -- (0.00000, 1.63299, 1.15470);
            \draw[edge,back] (0.00000, 1.63299, 1.15470) -- (-0.70711, 0.40825, 1.15470);
            \draw[edge,back] (0.00000, 1.63299, 1.15470) -- (0.70711, 1.22474, 0.00000);
            \draw[edge,back] (-0.70711, 0.40825, 1.15470) -- (0.00000, -0.81650, 1.15470);
            \draw[edge,back] (-0.70711, 0.40825, 1.15470) -- (0.00000, 0.00000, 0.00000);
            \draw[edge,back] (0.70711, 1.22474, 0.00000) -- (0.00000, 0.00000, 0.00000);
                \draw[subedge,subback] (0.00000, 1.22474, 1.73205) -- (1.41421, 1.22474, 1.73205);
                \draw[subedge,subback] (0.00000, 1.22474, 1.73205) -- (-0.70711, 0.40825, 1.15470);
                \draw[subedge,subback] (-0.70711, 0.40825, 1.15470) -- (0.00000, -0.40825, 0.57735);
            \node[vertex,] at (0.00000, 1.63299, 1.15470)     {};
            \node[vertex,] at (-0.70711, 0.40825, 1.15470)     {};
                \fill[subfacet] (0.00000, -0.40825, 0.57735) -- (-0.70711, 0.40825, 1.15470) -- (0.00000, 1.22474, 1.73205) -- (1.41421, 1.22474, 1.73205) -- (2.12132, 0.40825, 1.15470) -- (1.41421, -0.40825, 0.57735) -- cycle {};
                \node[subvertex] at (0.00000, 1.22474, 1.73205)     {};
                \node[subvertex] at (-0.70711, 0.40825, 1.15470)     {};
            \fill[facet] (0.70711, -0.40825, 2.30940) -- (1.41421, 0.81650, 2.30940) -- (0.00000, 0.81650, 2.30940) -- cycle {};
            \fill[facet] (2.12132, 0.40825, 1.15470) -- (1.41421, 0.81650, 2.30940) -- (1.41421, 1.63299, 1.15470) -- cycle {};
            \fill[facet] (1.41421, -0.81650, 1.15470) -- (0.70711, -0.40825, 2.30940) -- (1.41421, 0.81650, 2.30940) -- (2.12132, 0.40825, 1.15470) -- cycle {};
            \fill[facet] (0.00000, -0.81650, 1.15470) -- (0.70711, -0.40825, 2.30940) -- (1.41421, -0.81650, 1.15470) -- cycle {};
            \fill[facet] (1.41421, 0.00000, 0.00000) -- (2.12132, 0.40825, 1.15470) -- (1.41421, 1.63299, 1.15470) -- (0.70711, 1.22474, 0.00000) -- cycle {};
            \fill[facet] (1.41421, 0.00000, 0.00000) -- (2.12132, 0.40825, 1.15470) -- (1.41421, -0.81650, 1.15470) -- cycle {};
            \fill[facet] (0.00000, 0.00000, 0.00000) -- (0.00000, -0.81650, 1.15470) -- (1.41421, -0.81650, 1.15470) -- (1.41421, 0.00000, 0.00000) -- cycle {};
            \draw[edge] (1.41421, 0.81650, 2.30940) -- (0.00000, 0.81650, 2.30940);
            \draw[edge] (1.41421, 0.81650, 2.30940) -- (0.70711, -0.40825, 2.30940);
            \draw[edge] (1.41421, 0.81650, 2.30940) -- (1.41421, 1.63299, 1.15470);
            \draw[edge] (1.41421, 0.81650, 2.30940) -- (2.12132, 0.40825, 1.15470);
            \draw[edge] (0.00000, 0.81650, 2.30940) -- (0.70711, -0.40825, 2.30940);
            \draw[edge] (0.70711, -0.40825, 2.30940) -- (1.41421, -0.81650, 1.15470);
            \draw[edge] (0.70711, -0.40825, 2.30940) -- (0.00000, -0.81650, 1.15470);
            \draw[edge] (1.41421, 1.63299, 1.15470) -- (2.12132, 0.40825, 1.15470);
            \draw[edge] (1.41421, 1.63299, 1.15470) -- (0.70711, 1.22474, 0.00000);
            \draw[edge] (2.12132, 0.40825, 1.15470) -- (1.41421, -0.81650, 1.15470);
            \draw[edge] (2.12132, 0.40825, 1.15470) -- (1.41421, 0.00000, 0.00000);
            \draw[edge] (1.41421, -0.81650, 1.15470) -- (0.00000, -0.81650, 1.15470);
            \draw[edge] (1.41421, -0.81650, 1.15470) -- (1.41421, 0.00000, 0.00000);
            \draw[edge] (0.00000, -0.81650, 1.15470) -- (0.00000, 0.00000, 0.00000);
            \draw[edge] (0.70711, 1.22474, 0.00000) -- (1.41421, 0.00000, 0.00000);
            \draw[edge] (1.41421, 0.00000, 0.00000) -- (0.00000, 0.00000, 0.00000);
                \draw[subedge] (1.41421, 1.22474, 1.73205) -- (2.12132, 0.40825, 1.15470);
                \draw[subedge] (2.12132, 0.40825, 1.15470) -- (1.41421, -0.40825, 0.57735);
                \draw[subedge] (1.41421, -0.40825, 0.57735) -- (0.00000, -0.40825, 0.57735);
            \node[vertex,] at (1.41421, 0.81650, 2.30940)     {};
            \node[vertex,] at (0.00000, 0.81650, 2.30940)     {};
            \node[vertex,] at (0.70711, -0.40825, 2.30940)     {};
            \node[vertex,] at (1.41421, 1.63299, 1.15470)     {};
            \node[vertex,] at (2.12132, 0.40825, 1.15470)     {};
            \node[vertex,] at (1.41421, -0.81650, 1.15470)     {};
            \node[vertex,] at (0.00000, -0.81650, 1.15470)     {};
            \node[vertex,] at (0.70711, 1.22474, 0.00000)     {};
            \node[vertex,] at (1.41421, 0.00000, 0.00000)     {};
            \node[vertex,] at (0.00000, 0.00000, 0.00000)     {};
                \node[subvertex] at (1.41421, 1.22474, 1.73205)     {};
                \node[subvertex] at (2.12132, 0.40825, 1.15470)     {};
                \node[subvertex] at (1.41421, -0.40825, 0.57735)     {};
                \node[subvertex] at (0.00000, -0.40825, 0.57735)     {};
        \end{tikzpicture}
        \subcaption{$\SEP(K_4)^{(12)}$}
    \end{subfigure}%
    \begin{subfigure}{0.25\textwidth}
        \centering
        \begin{tikzpicture}[
            x={(0.374024cm, -0.375379cm)},
            y={(0.927419cm, 0.151412cm)},
            z={(-0.000023cm, 0.914421cm)},
            scale=1.0,
            ]
            \draw[edge,back] (0.00000, 0.81650, 2.30940) -- (0.00000, 1.63299, 1.15470);
            \draw[edge,back] (0.00000, 0.81650, 2.30940) -- (-0.70711, 0.40825, 1.15470);
            \draw[edge,back] (1.41421, 1.63299, 1.15470) -- (0.00000, 1.63299, 1.15470);
            \draw[edge,back] (0.00000, 1.63299, 1.15470) -- (-0.70711, 0.40825, 1.15470);
            \draw[edge,back] (0.00000, 1.63299, 1.15470) -- (0.70711, 1.22474, 0.00000);
            \draw[edge,back] (-0.70711, 0.40825, 1.15470) -- (0.00000, -0.81650, 1.15470);
            \draw[edge,back] (-0.70711, 0.40825, 1.15470) -- (0.00000, 0.00000, 0.00000);
            \draw[edge,back] (0.70711, 1.22474, 0.00000) -- (0.00000, 0.00000, 0.00000);
            \node[vertex,] at (0.00000, 1.63299, 1.15470)     {};
            \node[vertex,] at (-0.70711, 0.40825, 1.15470)     {};
                \node[subvertex] at (-0.23570, -0.13608, 0.76980) {};
                \draw[subedge,subback] (1.64992, 0.95258, 1.53960) -- (-0.23570, -0.13608, 0.76980);
            \fill[facet] (0.70711, -0.40825, 2.30940) -- (1.41421, 0.81650, 2.30940) -- (0.00000, 0.81650, 2.30940) -- cycle {};
            \fill[facet] (2.12132, 0.40825, 1.15470) -- (1.41421, 0.81650, 2.30940) -- (1.41421, 1.63299, 1.15470) -- cycle {};
            \fill[facet] (1.41421, -0.81650, 1.15470) -- (0.70711, -0.40825, 2.30940) -- (1.41421, 0.81650, 2.30940) -- (2.12132, 0.40825, 1.15470) -- cycle {};
            \fill[facet] (0.00000, -0.81650, 1.15470) -- (0.70711, -0.40825, 2.30940) -- (1.41421, -0.81650, 1.15470) -- cycle {};
            \fill[facet] (1.41421, 0.00000, 0.00000) -- (2.12132, 0.40825, 1.15470) -- (1.41421, 1.63299, 1.15470) -- (0.70711, 1.22474, 0.00000) -- cycle {};
            \fill[facet] (1.41421, 0.00000, 0.00000) -- (2.12132, 0.40825, 1.15470) -- (1.41421, -0.81650, 1.15470) -- cycle {};
            \fill[facet] (0.00000, 0.00000, 0.00000) -- (0.00000, -0.81650, 1.15470) -- (1.41421, -0.81650, 1.15470) -- (1.41421, 0.00000, 0.00000) -- cycle {};
            \draw[edge] (1.41421, 0.81650, 2.30940) -- (0.00000, 0.81650, 2.30940);
            \draw[edge] (1.41421, 0.81650, 2.30940) -- (0.70711, -0.40825, 2.30940);
            \draw[edge] (1.41421, 0.81650, 2.30940) -- (1.41421, 1.63299, 1.15470);
            \draw[edge] (1.41421, 0.81650, 2.30940) -- (2.12132, 0.40825, 1.15470);
            \draw[edge] (0.00000, 0.81650, 2.30940) -- (0.70711, -0.40825, 2.30940);
            \draw[edge] (0.70711, -0.40825, 2.30940) -- (1.41421, -0.81650, 1.15470);
            \draw[edge] (0.70711, -0.40825, 2.30940) -- (0.00000, -0.81650, 1.15470);
            \draw[edge] (1.41421, 1.63299, 1.15470) -- (2.12132, 0.40825, 1.15470);
            \draw[edge] (1.41421, 1.63299, 1.15470) -- (0.70711, 1.22474, 0.00000);
            \draw[edge] (2.12132, 0.40825, 1.15470) -- (1.41421, -0.81650, 1.15470);
            \draw[edge] (2.12132, 0.40825, 1.15470) -- (1.41421, 0.00000, 0.00000);
            \draw[edge] (1.41421, -0.81650, 1.15470) -- (0.00000, -0.81650, 1.15470);
            \draw[edge] (1.41421, -0.81650, 1.15470) -- (1.41421, 0.00000, 0.00000);
            \draw[edge] (0.00000, -0.81650, 1.15470) -- (0.00000, 0.00000, 0.00000);
            \draw[edge] (0.70711, 1.22474, 0.00000) -- (1.41421, 0.00000, 0.00000);
            \draw[edge] (1.41421, 0.00000, 0.00000) -- (0.00000, 0.00000, 0.00000);
            \node[vertex,] at (1.41421, 0.81650, 2.30940)     {};
            \node[vertex,] at (0.00000, 0.81650, 2.30940)     {};
            \node[vertex,] at (0.70711, -0.40825, 2.30940)     {};
            \node[vertex,] at (1.41421, 1.63299, 1.15470)     {};
            \node[vertex,] at (2.12132, 0.40825, 1.15470)     {};
            \node[vertex,] at (1.41421, -0.81650, 1.15470)     {};
            \node[vertex,] at (0.00000, -0.81650, 1.15470)     {};
            \node[vertex,] at (0.70711, 1.22474, 0.00000)     {};
            \node[vertex,] at (1.41421, 0.00000, 0.00000)     {};
            \node[vertex,] at (0.00000, 0.00000, 0.00000)     {};
                \node[subvertex] at (1.64992, 0.95258, 1.53960) {};
        \end{tikzpicture}
        \subcaption{$\SEP(K_4)^{(123)}$}
    \end{subfigure}%
    \begin{subfigure}{0.25\textwidth}
        \centering
        \begin{tikzpicture}[
            x={(0.374024cm, -0.375379cm)},
            y={(0.927419cm, 0.151412cm)},
            z={(-0.000023cm, 0.914421cm)},
            scale=1.0,
            ]
            \draw[edge,back] (0.00000, 0.81650, 2.30940) -- (0.00000, 1.63299, 1.15470);
            \draw[edge,back] (0.00000, 0.81650, 2.30940) -- (-0.70711, 0.40825, 1.15470);
            \draw[edge,back] (1.41421, 1.63299, 1.15470) -- (0.00000, 1.63299, 1.15470);
            \draw[edge,back] (0.00000, 1.63299, 1.15470) -- (-0.70711, 0.40825, 1.15470);
            \draw[edge,back] (0.00000, 1.63299, 1.15470) -- (0.70711, 1.22474, 0.00000);
            \draw[edge,back] (-0.70711, 0.40825, 1.15470) -- (0.00000, -0.81650, 1.15470);
            \draw[edge,back] (-0.70711, 0.40825, 1.15470) -- (0.00000, 0.00000, 0.00000);
            \draw[edge,back] (0.70711, 1.22474, 0.00000) -- (0.00000, 0.00000, 0.00000);
            \node[vertex,] at (0.00000, 1.63299, 1.15470)     {};
            \node[vertex,] at (-0.70711, 0.40825, 1.15470)     {};
                \node[subvertex] at (0.70711, 1.22474, 1.73205) {};
                \draw[subedge,subback] (0.70711, 1.22474, 1.73205) -- (0.70711, -0.40825, 0.57735);
            \fill[facet] (0.70711, -0.40825, 2.30940) -- (1.41421, 0.81650, 2.30940) -- (0.00000, 0.81650, 2.30940) -- cycle {};
            \fill[facet] (2.12132, 0.40825, 1.15470) -- (1.41421, 0.81650, 2.30940) -- (1.41421, 1.63299, 1.15470) -- cycle {};
            \fill[facet] (1.41421, -0.81650, 1.15470) -- (0.70711, -0.40825, 2.30940) -- (1.41421, 0.81650, 2.30940) -- (2.12132, 0.40825, 1.15470) -- cycle {};
            \fill[facet] (0.00000, -0.81650, 1.15470) -- (0.70711, -0.40825, 2.30940) -- (1.41421, -0.81650, 1.15470) -- cycle {};
            \fill[facet] (1.41421, 0.00000, 0.00000) -- (2.12132, 0.40825, 1.15470) -- (1.41421, 1.63299, 1.15470) -- (0.70711, 1.22474, 0.00000) -- cycle {};
            \fill[facet] (1.41421, 0.00000, 0.00000) -- (2.12132, 0.40825, 1.15470) -- (1.41421, -0.81650, 1.15470) -- cycle {};
            \fill[facet] (0.00000, 0.00000, 0.00000) -- (0.00000, -0.81650, 1.15470) -- (1.41421, -0.81650, 1.15470) -- (1.41421, 0.00000, 0.00000) -- cycle {};
            \draw[edge] (1.41421, 0.81650, 2.30940) -- (0.00000, 0.81650, 2.30940);
            \draw[edge] (1.41421, 0.81650, 2.30940) -- (0.70711, -0.40825, 2.30940);
            \draw[edge] (1.41421, 0.81650, 2.30940) -- (1.41421, 1.63299, 1.15470);
            \draw[edge] (1.41421, 0.81650, 2.30940) -- (2.12132, 0.40825, 1.15470);
            \draw[edge] (0.00000, 0.81650, 2.30940) -- (0.70711, -0.40825, 2.30940);
            \draw[edge] (0.70711, -0.40825, 2.30940) -- (1.41421, -0.81650, 1.15470);
            \draw[edge] (0.70711, -0.40825, 2.30940) -- (0.00000, -0.81650, 1.15470);
            \draw[edge] (1.41421, 1.63299, 1.15470) -- (2.12132, 0.40825, 1.15470);
            \draw[edge] (1.41421, 1.63299, 1.15470) -- (0.70711, 1.22474, 0.00000);
            \draw[edge] (2.12132, 0.40825, 1.15470) -- (1.41421, -0.81650, 1.15470);
            \draw[edge] (2.12132, 0.40825, 1.15470) -- (1.41421, 0.00000, 0.00000);
            \draw[edge] (1.41421, -0.81650, 1.15470) -- (0.00000, -0.81650, 1.15470);
            \draw[edge] (1.41421, -0.81650, 1.15470) -- (1.41421, 0.00000, 0.00000);
            \draw[edge] (0.00000, -0.81650, 1.15470) -- (0.00000, 0.00000, 0.00000);
            \draw[edge] (0.70711, 1.22474, 0.00000) -- (1.41421, 0.00000, 0.00000);
            \draw[edge] (1.41421, 0.00000, 0.00000) -- (0.00000, 0.00000, 0.00000);
            \node[vertex,] at (1.41421, 0.81650, 2.30940)     {};
            \node[vertex,] at (0.00000, 0.81650, 2.30940)     {};
            \node[vertex,] at (0.70711, -0.40825, 2.30940)     {};
            \node[vertex,] at (1.41421, 1.63299, 1.15470)     {};
            \node[vertex,] at (2.12132, 0.40825, 1.15470)     {};
            \node[vertex,] at (1.41421, -0.81650, 1.15470)     {};
            \node[vertex,] at (0.00000, -0.81650, 1.15470)     {};
            \node[vertex,] at (0.70711, 1.22474, 0.00000)     {};
            \node[vertex,] at (1.41421, 0.00000, 0.00000)     {};
            \node[vertex,] at (0.00000, 0.00000, 0.00000)     {};
                \node[subvertex] at (0.70711, -0.40825, 0.57735) {};
        \end{tikzpicture}
        \subcaption{$\SEP(K_4)^{(12)(34)}$}
    \end{subfigure}%
    \begin{subfigure}{0.25\textwidth}
        \centering
        \begin{tikzpicture}[
            x={(0.374024cm, -0.375379cm)},
            y={(0.927419cm, 0.151412cm)},
            z={(-0.000023cm, 0.914421cm)},
            scale=1.0,
            ]
            \draw[edge,back] (0.00000, 0.81650, 2.30940) -- (0.00000, 1.63299, 1.15470);
            \draw[edge,back] (0.00000, 0.81650, 2.30940) -- (-0.70711, 0.40825, 1.15470);
            \draw[edge,back] (1.41421, 1.63299, 1.15470) -- (0.00000, 1.63299, 1.15470);
            \draw[edge,back] (0.00000, 1.63299, 1.15470) -- (-0.70711, 0.40825, 1.15470);
            \draw[edge,back] (0.00000, 1.63299, 1.15470) -- (0.70711, 1.22474, 0.00000);
            \draw[edge,back] (-0.70711, 0.40825, 1.15470) -- (0.00000, -0.81650, 1.15470);
            \draw[edge,back] (-0.70711, 0.40825, 1.15470) -- (0.00000, 0.00000, 0.00000);
            \draw[edge,back] (0.70711, 1.22474, 0.00000) -- (0.00000, 0.00000, 0.00000);
            \node[vertex,] at (0.00000, 1.63299, 1.15470)     {};
            \node[vertex,] at (-0.70711, 0.40825, 1.15470)     {};
                \node[subvertex] at (0.70711, 0.40825, 1.15470) {};
            \fill[facet] (0.70711, -0.40825, 2.30940) -- (1.41421, 0.81650, 2.30940) -- (0.00000, 0.81650, 2.30940) -- cycle {};
            \fill[facet] (2.12132, 0.40825, 1.15470) -- (1.41421, 0.81650, 2.30940) -- (1.41421, 1.63299, 1.15470) -- cycle {};
            \fill[facet] (1.41421, -0.81650, 1.15470) -- (0.70711, -0.40825, 2.30940) -- (1.41421, 0.81650, 2.30940) -- (2.12132, 0.40825, 1.15470) -- cycle {};
            \fill[facet] (0.00000, -0.81650, 1.15470) -- (0.70711, -0.40825, 2.30940) -- (1.41421, -0.81650, 1.15470) -- cycle {};
            \fill[facet] (1.41421, 0.00000, 0.00000) -- (2.12132, 0.40825, 1.15470) -- (1.41421, 1.63299, 1.15470) -- (0.70711, 1.22474, 0.00000) -- cycle {};
            \fill[facet] (1.41421, 0.00000, 0.00000) -- (2.12132, 0.40825, 1.15470) -- (1.41421, -0.81650, 1.15470) -- cycle {};
            \fill[facet] (0.00000, 0.00000, 0.00000) -- (0.00000, -0.81650, 1.15470) -- (1.41421, -0.81650, 1.15470) -- (1.41421, 0.00000, 0.00000) -- cycle {};
            \draw[edge] (1.41421, 0.81650, 2.30940) -- (0.00000, 0.81650, 2.30940);
            \draw[edge] (1.41421, 0.81650, 2.30940) -- (0.70711, -0.40825, 2.30940);
            \draw[edge] (1.41421, 0.81650, 2.30940) -- (1.41421, 1.63299, 1.15470);
            \draw[edge] (1.41421, 0.81650, 2.30940) -- (2.12132, 0.40825, 1.15470);
            \draw[edge] (0.00000, 0.81650, 2.30940) -- (0.70711, -0.40825, 2.30940);
            \draw[edge] (0.70711, -0.40825, 2.30940) -- (1.41421, -0.81650, 1.15470);
            \draw[edge] (0.70711, -0.40825, 2.30940) -- (0.00000, -0.81650, 1.15470);
            \draw[edge] (1.41421, 1.63299, 1.15470) -- (2.12132, 0.40825, 1.15470);
            \draw[edge] (1.41421, 1.63299, 1.15470) -- (0.70711, 1.22474, 0.00000);
            \draw[edge] (2.12132, 0.40825, 1.15470) -- (1.41421, -0.81650, 1.15470);
            \draw[edge] (2.12132, 0.40825, 1.15470) -- (1.41421, 0.00000, 0.00000);
            \draw[edge] (1.41421, -0.81650, 1.15470) -- (0.00000, -0.81650, 1.15470);
            \draw[edge] (1.41421, -0.81650, 1.15470) -- (1.41421, 0.00000, 0.00000);
            \draw[edge] (0.00000, -0.81650, 1.15470) -- (0.00000, 0.00000, 0.00000);
            \draw[edge] (0.70711, 1.22474, 0.00000) -- (1.41421, 0.00000, 0.00000);
            \draw[edge] (1.41421, 0.00000, 0.00000) -- (0.00000, 0.00000, 0.00000);
            \node[vertex,] at (1.41421, 0.81650, 2.30940)     {};
            \node[vertex,] at (0.00000, 0.81650, 2.30940)     {};
            \node[vertex,] at (0.70711, -0.40825, 2.30940)     {};
            \node[vertex,] at (1.41421, 1.63299, 1.15470)     {};
            \node[vertex,] at (2.12132, 0.40825, 1.15470)     {};
            \node[vertex,] at (1.41421, -0.81650, 1.15470)     {};
            \node[vertex,] at (0.00000, -0.81650, 1.15470)     {};
            \node[vertex,] at (0.70711, 1.22474, 0.00000)     {};
            \node[vertex,] at (1.41421, 0.00000, 0.00000)     {};
            \node[vertex,] at (0.00000, 0.00000, 0.00000)     {};
        \end{tikzpicture}
        \subcaption{$\SEP(K_4)^{(1234)}$}
    \end{subfigure}
    \caption{Fixed polytopes (orange) of $\SEP(K_4)$ (blue) fixed by permutations acting by coordinate permutation.}
    \label{fig:sep-k4-fixed}
\end{figure*}
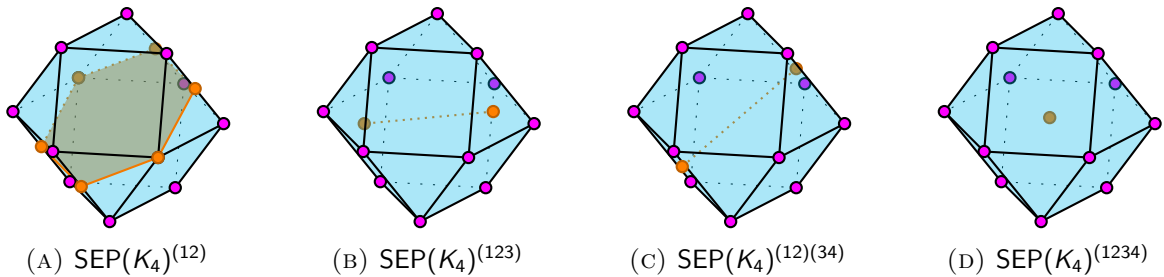

Additionally, in our discussion of the fixed polytopes of symmetric edge polytopes, we must first give attention to the set of actions under which these polytopes are invariant.
To this end, we define the restricted automorphism group of a polytope, which can intuitively be thought of as the group of distinct linear actions under which a polytope is invariant.
In Section \ref{sec:invariant_actions}, we investigate the restricted automorphism groups of symmetric edge polytopes, providing a partial description thereof in terms of the automorphism groups of the associated graphs, which we also define here.

\begin{definition}[Automorphism group of a graph]\label{def:auto-graph}
    Let $G$ be an unweighted graph with vertex set $\{v_1,\dots,v_n\}$ and edge set $E$, where $(v_i,v_j)\in E$ if $G$ contains an edge from $v_i$ to $v_j$.
    Then, the \defterm{automorphism group of $G$}, denoted $\Aut(G)$, is the subgroup of $S_{|V|}$ consisting of permutations $\sigma$ for which $(v_i,v_j)\in E$ if and only if $(v_{\sigma(i)},v_{\sigma(j)})\in E$.
\end{definition}

\begin{definition}[Restricted automorphism group of a polytope]\label{def:auto-polytope}
    Let $\P$ be a polytope in $\R^n$ with vertices $\{\vv_1,\dots,\vv_m\}$.
    Then, the \defterm{restricted automorphism group of $\P$}, denoted $\Aut(\P)$, is the group of distinct bijective maps on the points of $\P$ that are induced by the action of elements of $GL(\R^n)$ restricted to $\P$.

\end{definition}

Finally, several of our results describe fixed subpolytopes in terms of contractions of the graph associated with a symmetric edge polytope, and we define this concept here.

\begin{definition}[Contraction of a graph]\label{def:vertex-contraction}
    Let $G=(V,E)$ be a graph and let $S=\{S_1,\dots,S_m\}$ be a partition of the vertices of $G$.
    Then, $G'=(S,E')$ where $(S_i,S_j)\in E'$ if $(v_i,v_j)\in E$ for some $v_i\in S_i$ and $v_j\in S_j$ is the \defterm{contraction of $G$} induced by the partition $S$.
\end{definition}

\section{Invariant Actions on Symmetric Edge Polytopes} \label{sec:invariant_actions}

The majority of our paper will discuss the fixed polytopes of symmetric edge polytopes under coordinate permutations.
First, however, we identify the coordinate permutations under which symmetric edge polytopes are invariant, in terms of the associated graph, presented in Proposition \ref{prop:sep-perm}.
We also extend these results to obtain a full description of the restricted automorphism group of the symmetric edge polytope of the complete graph on $n$ vertices, presented in Proposition \ref{prop:sep-kn-auto}.

Additionally, we show in Theorem \ref{thm:sep-unitary-equiv} that symmetric edge polytopes are unitarily equivalent if and only if their associated graphs are isomorphic, in which case the equivalence can be described as a coordinate permutation.
Consequently, we show in Corollary \ref{cor:sep-unitary-auto} that the coordinate permutations under which a symmetric edge polytope is invariant, together with their negations, form the complete set of unitary symmetries of the polytope.

\begin{proposition}\label{prop:sep-perm}
    Let $G=([n],E)$ be a simple graph on $n$ vertices, and let $\sigma\in S_n$ act on $G$ by vertex permutation and on $\SEP(G)$ by coordinate permutation.
    Then, $\SEP(G)$ is invariant under $\sigma$ if and only if $G$ is invariant under $\sigma$.
\end{proposition}
\begin{proof}
    First, we show that $\SEP(G)$ is invariant under $\sigma$ if $G$ is invariant under $\sigma$.
    Assume $G$ is invariant under the vertex permutation action of $\sigma$.
    Then, for any edge $(i,j)\in E$, $(\sigma(i),\sigma(j))\in E$ as well.
    Thus, for every vertex $\ve_i-\ve_j$ of $\SEP(G)$, its image, $\ve_{\sigma(i)}-\ve_{\sigma(j)}$, is also a vertex of $\SEP(G)$.

    Similarly, for every vertex $\ve_{\sigma(i)}-\ve_{\sigma(j)}$ of $\sigma\cdot\SEP(G)$, its preimage, $\ve_{\sigma^{-1}(\sigma(i))}-\ve_{\sigma^{-1}(\sigma(j))}=\ve_i-\ve_j$, is a vertex of $\SEP(G)$.
    Thus, the action of $\sigma$ is a bijection on the vertices of $\SEP(G)$, so $\SEP(G)$ is invariant under the action of $\sigma$.

    Now, we show conversely that $G$ is invariant under $\sigma$ if $\SEP(G)$ is invariant under $\sigma$.
    Assume $\SEP(G)$ is invariant under the action of $\sigma$.
    Then, for any vertex $\ve_i-\ve_j$ of $\SEP(G)$, its image under $\sigma$, $\ve_{\sigma(i)}-\ve_{\sigma(j)}$, is also a vertex of $\SEP(G)$, and vice versa.
    Thus, $(\sigma(i),\sigma(j))$ is an edge in $G$ whenever $(i,j)$ is an edge in $G$, and vice versa.
    This implies that $G$ is invariant under $\sigma$ if $\SEP(G)$ is invariant under $\sigma$.
\end{proof}

To obtain a complete description of the restricted automorphism group of the symmetric edge polytope of a complete graph, we will consider not only the vertices of the polytope, but also its facets.
To this end, we leverage a description of the facets of the symmetric edge polytope of the complete graph on $n$ vertices provided by Ardila \etal in Proposition 11 of \cite{siam-root}.
We restate this result in the language of symmetric edge polytopes here.

\begin{proposition}\label{prop:11}
    $\SEP(K_n)$ is an $(n-1)$-dimensional polytope in $\R^n$ contained within the hyperplane $\set{\vx\in\R^n}{{\bv1}^\top\vx=0}$.
    It has $2^n-2$ facets, which can be labeled by nonempty proper subsets of $[n]$ such that the facets with label $S\subset[n]$ is contained within the hyperplane $\set{\vx\in\R^n}{\ve_S^\top\vx=1}$.
\end{proposition}

\begin{example}
    Consider the symmetric edge polytope associated with $K_{3}$ in Figure \ref{fig:sep-k3-labels}.
    Notice that $\text{dim}(\SEP(K_{3}))=2$ and there are exactly $6$ facet labels, each satisfying $S \subset [3]$.
\end{example}

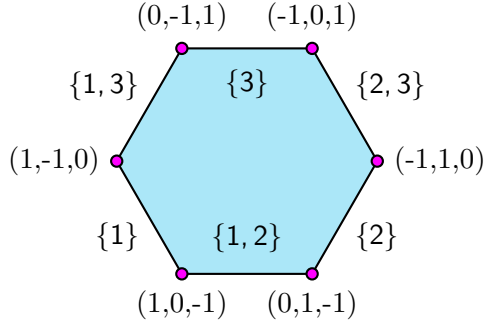
\begin{figure}
    \centering
    \begin{tikzpicture}[
                x={(-0.866cm,-0.5cm)},
                y={(0.866cm,-0.5cm)},
                z={(0cm,1cm)},
                scale = 1.0
                ]
                \coordinate (ab) at (1,-1,0);
                \coordinate (ac) at (1,0,-1);
                \coordinate (ba) at (-1,1,0);
                \coordinate (bc) at (0,1,-1);
                \coordinate (ca) at (-1,0,1);
                \coordinate (cb) at (0,-1,1);
                \fill[facet] (ab) -- (ac) -- (bc) -- (ba) -- (ca) -- (cb) -- cycle {};
                \draw[edge] (ab) -- (ac) node[pos=0.67, label={[black]left:$\{1\}$}] {};
                \draw[edge] (ac) -- (bc) node[pos=0.5, label={[black]above:$\{1,2\}$}] {};
                \draw[edge] (bc) -- (ba) node[pos=0.33, label={[black]right:$\{2\}$}] {};
                \draw[edge] (ba) -- (ca) node[pos=0.67, label={[black]right:$\{2,3\}$}] {};
                \draw[edge] (ca) -- (cb) node[pos=0.5, label={[black]below:$\{3\}$}] {};
                \draw[edge] (cb) -- (ab) node[pos=0.33, label={[black]left:$\{1,3\}$}] {};
                \node[vertex,label={[black]left:(1,-1,0)}] at (ab) {};
                \node[vertex,label={[black]below:(1,0,-1)}] at (ac) {};
                \node[vertex,label={[black]right:(-1,1,0)}] at (ba) {};
                \node[vertex,label={[black]below:(0,1,-1)}] at (bc) {};
                \node[vertex,label={[black]above:(-1,0,1)}] at (ca) {};
                \node[vertex,label={[black]above:(0,-1,1)}] at (cb) {};
            \end{tikzpicture}
    \caption{$\SEP(K_3)$ with facets labeled by nonempty proper subsets of $\{1,2,3\}$.}
    \label{fig:sep-k3-labels}
\end{figure}

We can further describe exactly which facets of $\SEP(K_n)$ are adjacent to each other.

\begin{lemma}\label{lemma:adjacent-dif-1}
    Two facets of $\SEP(K_n)$ are a common neighbor if and only if their labels, as described in Proposition \ref{prop:11}, differ in exactly one element.
\end{lemma}
\begin{proof}
    We say two facets of a $d$-dimensional polytope are adjacent if their intersection is a $(d-2)$-dimensional face.
    Consider two facets of $\SEP(K_n)$ with labels $S_1$ and $S_2$.
    Note that the facet with label $S_1$ is the convex hull of the vertices $\set{\ve_i-\ve_j}{i\in S_1,j\in S_1^C}$ and $S_2$ is the convex hull of the vertices $\set{\ve_i-\ve_j}{i\in S_2,j\in S_2^C}$.
    Thus, their intersection is the convex hull of the vertices $\set{\ve_i-\ve_j}{i\in S_1\cap S_2,j\in S_1^C\cap S_2^C}$.
    Applying Theorem 1 of \cite{pos-root}, we can see that this intersection has dimension $|S_1\cap S_2|+|S_1^C\cap S_2^C|-2$, so the intersection of the facets of $\SEP(K_n)$ with labels $S_1$ and $S_2$ has dimension $d-2=n-3$ if and only if $|S_1\cap S_2|+|S_1^C\cap S_2^C|=d=n-1$.

    Note that $|S_1\cap S_2^C|+|S_1^C\cap S_2|$ is the size of the differences between $S_1$ and $S_2$, and that \[|S_1\cap S_2|+|S_1^C\cap S_2^C|+|S_1\cap S_2^C|+|S_1^C\cap S_2|=n,\] so \[|S_1\cap S_2|+|S_1^C\cap S_2^C|=n-(|S_1\cap S_2^C|+|S_1^C\cap S_2|).\]
    Hence, $|S_1\cap S_2|+|S_1^C\cap S_2^C|=n-1$ if and only if the difference between $S_1$ and $S_2$ has size $1$, and so the facets of $\SEP(K_n)$ with labels $S_1$ and $S_2$ are adjacent if and only if $S_1$ and $S_2$ differ in exactly one element.
\end{proof}

These facts allows us to prove a lemma restricting the possible symmetries of $\SEP(K_n)$.

\begin{lemma}\label{lemma:size-1-labels}
    Consider the polytope $\SEP(K_n)$ for $n>2$.
    Label the facets of $\SEP(K_n)$ as described in Proposition \ref{prop:11}, and let $A\in GL(\R^n)$ act linearly on $\SEP(K_n)$.
    The action of $A$ on $\SEP(K_n)$ must either map all facets with labels of size 1 to facets with labels of size 1, or it must map all facets with labels of size 1 to facets with labels of size $n-1$.
    Further, the action of $A$ on $\SEP(K_n)$ is uniquely determined by its action on the facets of $\SEP(K_n)$ with labels of size 1.
\end{lemma}
\begin{proof}
    Note that a facet with label $S$ has exactly $|S||S^C|$ vertices, namely $\ve_i-\ve_j$ for $i\in S$ and $j\in S^C$.
   
    In particular, there are $2n$ facets of $\SEP(K_n)$ with exactly $n-1$ vertices when $n>2$, corresponding to the labels of size $1$ and $n-1$, and any linear action $A$ on $\R^n$ under which $\SEP(K_n)$ is invariant must induce a bijection on the set of these $2n$ facets.
    
    Further, from Lemma \ref{lemma:adjacent-dif-1}, facets of $\SEP(K_n)$ are adjacent if and only if their labels differ in exactly one element.
   
    Consequently the only common neighbor of the facets with labels $S$ and $S'$ of size $1$ is the facet with label $S\cup S'$, and the only common neighbor of the facets with labels $S$ and $S'$ of size $n-1$ is the facet with label $S\cap S'$.
    Meanwhile, no pair of facets with labels of sizes $1$ and $n-1$ share exactly one common neighbor; if $n\ne4$, such facets share no common neighbors, while for $n=4$, they share two common neighbors.
    
    Because linear transformation preserves the adjacency between facets, thus $A$ must map all of the facets with labels of size 1 to facets with labels of size 1, or $A$ must map all of the facets with labels of size 1 to facets with labels of size $n-1$.
    
    Note that the average of the vertices of a facet with label $S=\{i\}$ is
    \begin{equation*}
        \ve_S-\frac{1}{n-1}\ve_{S^C} = \frac{n}{n-1}\ve_i-\frac{1}{n-1}\bv1.
    \end{equation*}
    Further, any vertex $\ve_i-\ve_j$ of $\SEP(K_n)$ can be written as a linear combination of these vectors since
    \begin{equation*}
        \frac{n-1}{n}\left(\frac{n}{n-1}\ve_i-\frac{1}{n-1}\bv1\right)-\frac{n-1}{n}\left(\frac{n}{n-1}\ve_j-\frac{1}{n-1}\bv1\right) = \ve_i-\ve_j.
    \end{equation*}
    Thus, the action of $A$ on the span of $\SEP(K_n)$ is determined by its action on these facets.
\end{proof}

We are now ready to fully describe the structure of the restricted automorphism group of $\SEP(K_n)$.
We do so in Proposition \ref{prop:sep-kn-auto}, by finding a set of distinct symmetries of $\SEP(K_n)$ which exhausts the possibilities provided by Lemma \ref{lemma:size-1-labels}.

\begin{proposition}\label{prop:sep-kn-auto}
    The restricted automorphism group of $\SEP(K_n)$ is isomorphic to $S_n\times S_2$ for $n>2$, and is isomorphic to $S_2$ for $n=2$.
\end{proposition}
\begin{proof}
    When $n=2$, $\SEP(K_2)$ is the line segment from $(1,-1)$ to $(-1,1)$, which is readily seen to have a restricted automorphism group isomorphic to $S_2$.
    
    From Lemma \ref{lemma:size-1-labels}, we know each linear automorphism on $\SEP(K_n)$ can be described either by a bijection from the set of facets with labels of size 1 to the set of facets with labels of size 1, or a bijection from the set of facets with labels of size 1 to the set of facets with labels of size $(n-1)$.
    There are $n!$ bijections from the set of facets with labels of size $1$ to itself, and another $n!$ bijections from the set of facets with labels of size $1$ to facets with labels of size $n-1$.
    Thus, there are at most $2(n!)$ possible distinct linear actions on $\SEP(K_n)$.
    
    From Proposition \ref{prop:sep-perm}, we know that the restricted automorphism group of $\SEP(K_n)$ has a subgroup of order $n!$ corresponding to the set of all coordinate permutations since $K_n$ is invariant under vertex permutation.
    We can also see that $\SEP(K_n)$ is invariant under the action of $-I$, the negative identity transformation on $\R^n$, which maps each vertex $\ve_i-\ve_j$ of $\SEP(K_n)$ to the (distinct) vertex $\ve_j-\ve_i$.
    Further, we can show that the automorphism on $\SEP(K_n)$ induced by $-I$ is distinct from the automorphism induced by any coordinate permutation.
    
    Assume to the contrary that $-I$ acts on $\SEP(K_n)$ identically to the permutation matrix associated to some permutation $\sigma\in S_n$.
    Observe that $-I$ maps $\ve_i-\ve_j$ to $\ve_j-\ve_i$ for each $i\ne j$, so the cycle decomposition of $\sigma$ contains a 2-cycle $(i\ j)$ for each $i\ne j$.
    In particular, for $n>2$, the cycle decomposition contains the $2$-cycles $(12)$ and $(23)$, which is impossible, as these cycles are not disjoint.
    Thus, there can be no such $\sigma\in S_n$, and so $-I$ induces a distinct automorphism on $\SEP(K_n)$.

    Likewise, the negation of each coordinate permutation $P$, given by $(-I)P=P(-I)=-P$, also induces a distinct automorphism on $\SEP(K_n)$, providing $n!$ negated permutation automorphisms in addition to the $n!$ coordinate permutation automorphisms.
    This must be the complete restricted automorphism group of $\SEP(K_n)$, since it includes $2(n!)$ distinct automorphisms, which is the maximum number possible size of this group.
    Finally, since $\{I,-I\}$ and $S_n$ are both normal in the restricted automorphism group of $\SEP(K_n)$, and their free product is exactly the restricted automorphism group of $\SEP(K_n)$, and their intersection is solely the identity automorphism, we have that the restricted automorphism group of $\SEP(K_n)$ is isomorphic to $S_n\times S_2$.
\end{proof}

To obtain a similar result for symmetric edge polytopes of general connected graphs, we restrict our attention to unitary (rigid) elements of the restricted automorphism group.
Along the way, we also show that symmetric edge polytopes are unitarily equivalent exactly when their associated graphs are isomorphic.
This result complements the finding of D'Al\`i, Juhnke-Kubitzke, and Koch that symmetric edge polytopes are unimodularly equivalent exactly when their associated graphic matroids are isomorphic, even if the graphs themselves are not isomorphic \cite{DAliJuhnkeKoch}.
Further, our result can be used to provide an efficient reduction from the graph isomorphism problem-- which is conjectured to be NP-intermediate, i.e. neither P nor NP-complete \cite{graph-iso}-- to the problem of determining whether two symmetric edge polytopes are unitarily equivalent.
First, we prove a key lemma about the lattice spanned by the vertices of a symmetric edge polytope.

\begin{lemma}\label{lemma:sep-lattice}
    Let $G=([n],E)$ be a simple, connected graph with $n$ vertices.
    The vertices of $\SEP(G)$ define the $(n-1)$-dimensional lattice $\Lambda=\set{\vz\in\Z^n}{\langle\vz,\bv1\rangle=0}$.
\end{lemma}
\begin{proof}
    First, observe that $\vv\in\Lambda$ for every vertex $\vv$ of $\SEP(G)$ since $\SEP(G)$ is a lattice polytope and $\vv=\ve_i-\ve_j$ for some $i,j\in[n]$, so $\langle\vv,\bv1\rangle=\langle\ve_i,\bv1\rangle-\langle\ve_j,\bv1\rangle=1-1=0$.
    Thus, the lattice defined by the vertices of $\SEP(G)$ is contained within the lattice $\Lambda$.
    We will now show the reverse containment by expressing an arbitrary vector $\vz\in\Lambda$ as an integer linear combination of the vertices of $\SEP(G)$.

    Let $T_0$ be a spanning tree of $G$ with (arbitrary) root $r\in[n]$, and let $\vz_0=\bv0$.
    Then, for each $k\in\{0,\dots,n-1\}$, choose a leaf $i_k$ of $T_k$, let $j_k$ be the parent of $i_k$ in $T_k$, and define $T_{k+1}$ as $T_k$ with the vertex $i_k$ removed, and $\vz_{k+1}=\vz_k+\lambda_k(\ve_{i_k}-\ve_{j_k})$, where $\lambda_k$ is the $i_k$th coordinate of $\vz$ less the $i_k$th coordinate of $\vz_k$.
    Note that $\vz_k$ matches $\vz$ in coordinate $i_k$ and, since $i_k$ is not present in $T_{k'}$ for $k'>k$, $\vz_k$ matches $\vz$ in each of the coordinates $i_1,\dots,i_k$.
    Thus, $\vz_n$ matches $\vz$ in all but one coordinate, and since $\langle\vz,\bv1\rangle=0$, this ensures $\vz_n=\vz$.
    Further, $\vz_n$ is a sum of integer multiples of vertices of $\SEP(G)$, so $\vz_n=\vz$ is contained in the lattice defined by the vertices of $\SEP(G)$, so this lattice contains $\Lambda$.
\end{proof}

\begin{theorem}\label{thm:sep-unitary-equiv}
    Let $G$ and $G'$ be simple, connected graphs.
    Then, the following are equivalent:
    \begin{enumerate}
        \item $G$ is isomorphic to $G'$.
        \item $\SEP(G)$ is equivalent to $\SEP(G')$ via coordinate permutation.
        \item $\SEP(G)$ is unitarily equivalent to $\SEP(G')$.
    \end{enumerate}
    Further, if $\SEP(G)$ is unitarily equivalent to $\SEP(G')$ via a unitary transformation $U$, then $U$ is either a coordinate permutation or the negation of a coordinate permutation.
\end{theorem}
\begin{proof}
    First, we will show that $G$ is isomorphic to $G'$ if and only if $\SEP(G)$ is equivalent to $\SEP(G')$ by coordinate permutation.
    If $G$ is isomorphic to $G'$, then this isomorphism can be defined by some $\sigma\in S_n$ such that vertex $i$ of $G$ is mapped to vertex $\sigma(i)$ of $G'$, so $(i,j)$ is an edge in $G$ if and only if $(\sigma(i),\sigma(j))$ is an edge in $G'$.
    Thus the coordinate transformation $P_\sigma$ defined by $P_\sigma\ve_i=\ve_{\sigma(i)}$ maps $\SEP(G)$ to $\SEP(G')$.
    Conversely, if the coordinate transformation $P_\sigma$ maps $\SEP(G)$ to $\SEP(G')$, then $(i,j)$ is an edge of $G$ if and only if $(\sigma(i),\sigma(j))$ is an edge of $G'$, so $G$ is isomorphic to $G'$ via the mapping taking vertex $i$ of $G$ to vertex $\sigma(i)$ of $G'$.

    Next, recall that coordinate permutations are unitary, so if $\SEP(G)$ is equivalent to $\SEP(G')$ via coordinate permutation, then $\SEP(G)$ is unitarily equivalent to $\SEP(G')$.
    To prove the converse, we will show that any unitary transformation $U$ mapping $\SEP(G)$ to $\SEP(G')$ must be either a coordinate permutation or the negation thereof.
    Since $\SEP(G)$ and $\SEP(G')$ are invariant under negation, this is sufficient to show that unitary equivalence between $\SEP(G)$ and $\SEP(G')$ implies their equivalence via a coordinate permutation.

    Let $U$ be a unitary transformation mapping $\SEP(G)$ to $\SEP(G')$.
    First, we argue that $U\bv1=\pm\bv1$.
    Assume to the contrary that $U\bv1\ne\pm\bv1$.
    Then, $U^\top\bv1\ne\pm\bv1$ since otherwise $\bv1=UU^\top\bv1=\pm U\bv1\ne\pm\bv1$.
    So, noting that $\Span\SEP(G)=(\Span\{\bv1\})^\bot$, let $U^\top\bv1=\alpha\bv1+\beta\vx$ for some nonzero $\vx\in\SEP(G)$.
    Note that $\left\langle U\vx,\bv1\right\rangle=\vx^\top U^\top\bv1=\vx^\top(\alpha\bv1+\beta\vx)=\beta\vx^\top\vx\ne0$, so $U\vx\notin\SEP(G)$, so $\SEP(G)$ is not invariant under $U$.
    Thus, $U\bv1=\pm\bv1$.
    
    Then, Lemma \ref{lemma:sep-lattice} ensures $U$ maps points in the lattice $\Lambda=\set{\vz\in\Z^n}{\langle\vz,\bv1\rangle=0}$ to points in $\Lambda$.
    Since $U$ is unitary, lengths are preserved, so $\ve_i-\ve_j$ is mapped to $\ve_k-\ve_l$ for some $k,l\in[n]$ and each $i,j\in[n]$.
    Further, we claim that for every $i\in[n]$, the images of every vector $\ve_i-\ve_j$ for $j\in[n]\setminus\{i\}$ all have the same nonzero value in some coordinate $k$.
    It suffices to show that $U(\ve_i-\ve_j)$, $U(\ve_i-\ve_k)$, and $U(\ve_i-\ve_l)$ all share the same value in some nonzero coordinate for arbitrary distinct $i,j,k,l\in[n]$.
    First, note that the inner product of any pair of these three vectors is $1$ because $U$ is unitary, so each pair of these vectors must match in the value of exactly one nonzero coordinate.
    Let $U(\ve_i-\ve_j)$ match $U(\ve_i-\ve_k)$ in coordinate $c$ with value $v=\pm1$ and let $U(\ve_i-\ve_j)$ match $U(\ve_i-\ve_l)$ in coordinate $c'$ with value $v'=\pm1$.
    Then, $U(\ve_i-\ve_k)$ must match $U(\ve_i-\ve_l)$ in some coordinate $c^\pprime$ with value $v^\pprime=\pm1$.
    If $v^\pprime=v$, then $c^\pprime=c$ since this is the only coordinate of $U(\ve_i-\ve_k)$ with value $v$, so $U(\ve_i-\ve_j)$, $U(\ve_i-\ve_k)$, and $U(\ve_i-\ve_l)$ all match in coordinate $c$ with value $v$.
    Otherwise, $v^\pprime=v'$, so $c^\pprime=c'$ since this the only coordinate of $U(\ve_i-\ve_l)$ with value $v'$, so $U(\ve_i-\ve_j)$, $U(\ve_i-\ve_k)$, and $U(\ve_i-\ve_l)$ all match in coordinate $c'$ with value $v'$.

    Now, define $W=U$ if $U\bv1=\bv1$ or else $W=-U$ if $U\bv1=-\bv1$.
    Since $\SEP(G')$ is invariant under negation, $W=\pm U$ maps $\SEP(G)$ to $\SEP(G')$.
    Observe for each $i\in[n]$ that
    \begin{gather*}
        \ve_i=\frac{1}{n}\left(\bv1+\sum_{j\in[n]\setminus\{i\}}\ve_i-\ve_j\right),
        \qso\\
        W\ve_i=\frac{1}{n}\left(W\bv1+\sum_{j\in[n]\setminus\{i\}}W(\ve_i-\ve_j)\right)=\ve_k
    \end{gather*}
    for some $k\in[n]$.
    Since $W=\pm U$ is invertible, each $\ve_i$ must be mapped to a distinct $\ve_k$.
    Thus, $W=\pm U$ is a coordinate permutation, so $U$ is either a coordinate permutation or a negation thereof.
\end{proof}

Noting that every graph is trivially isomorphic to itself, we can leverage Theorem \ref{thm:sep-unitary-equiv} to describe the set of rigid symmetries of a symmetric edge polytope.

\begin{corollary}\label{cor:sep-unitary-auto}
    Let $G$ be a connected simple graph on $n$ vertices and $U$ be an orthogonal (unitary) linear transformation on $\R^n$.
    Then, $\SEP(G)$ is invariant under the action of $U$ if and only if $U$ is a coordinate permutation corresponding to a permutation of the vertices of $G$ under which $G$ is invariant, or the negation thereof.
\end{corollary}
\begin{proof}
    Take $G'=G$ in Theorem \ref{thm:sep-unitary-equiv}.
    A unitary transformation $U$ maps $\SEP(G)$ to itself only if $U$ is a coordinate permutation or the negation thereof, and Proposition \ref{prop:sep-perm} ensures the relevant coordinate permutation must correspond to a vertex permutation of $G$ under which $G$ is invariant.
    Conversely, Proposition \ref{prop:sep-perm} ensures $\SEP(G)$ is invariant under such coordinate permutations and their negations, and these transformations are easily seen to be unitary.
\end{proof}

To understand invariance of $\SEP(K_n)$ under $S_n$-action, we also examine the orbit structure of its $k$th dilate. Our next result provides insight on the stability of $S_n$ orbit counts, which provides insights for future work on the closed form for $\rvol(k\SEP(K_n))$.

\begin{theorem}\label{thm:orbit-count}
 Let $L_{n,k}=k\SEP(K_n)\cap\Z^n$ and $O_{n,k}$ be the number of $S_n$-orbits of $L_{n,k}$.
    Then, $O_{n,k}=\sum_{k'=0}^k\mathrm{p}(k')^2$ for $n \geq 2k$, where $\mathrm{p}$ is the partition function.
\end{theorem}
\begin{proof}
    Define $\tilde{L}_{n,k}=L_{n,k}\setminus L_{n,k-1}$ when $k>0$ and $\tilde{L}_{n,0}=L_{n,0}$.
    Further, let $\tilde{O}_{n,k}=|\tilde{L}_{n,k}/S_n|$.
    Then, the orbits of points in $\tilde{L}_{n,k}$ are entirely contained in this set, so $O_{n,k}=\sum_{k'=0}^k \tilde{O}_{n,k'}$.
    A point $\vz\in\Z^n$ is in $\tilde{L}_{n,k}$ if and only if the sum of its positive coordinates is $k$ and the sum of its negative coordinates is $-k$, and two such points are in the same $S_n$-orbit if and only if the sets of their coordinates are equal.
    There are $\mathrm{p}(k)$ sets of positive integers summing to $k$, so there are $\mathrm{p}(k)$ sets of positive numbers and $\mathrm{p}(k)$ sets of negative numbers which could form the set of positive coordinates and negative coordinates, respectively, of a point in $L_{n,k}$.
    Provided that $n\ge2k$ so that points in $\tilde{L}_{n,k}$ have enough coordinates to accommodate all pairs of partitions of $k$, the choice of these sets is independent, and uniquely describes the orbit of a point in $\tilde{L}_{n,k}$.
    Thus, there are $\tilde{O}_{n,k}=\mathrm{p}(k)^2$ for all $n\ge2k$, so $O_{n,k}=\sum_{k'=0}^k\tilde{O}_{n,k'}=\sum_{k'=0}^k \mathrm{p}(k')^2$ for $n\ge2k$.
\end{proof}

\section{Fixed Polytopes of Symmetric Edge Polytopes}\label{sec:fixed_polytopes}

In this section, we provide vertex descriptions of the fixed polytopes of symmetric edge polytopes fixed under coordinate permutation.
We further show that these fixed polytopes are combinatorially and, in fact, linearly equivalent to symmetric edge polytopes corresponding to contractions of the original graph.

First, we prove a pair of lemmas used in our subsequent results.

\begin{lemma}\label{lemma:mapped-vertices}
    Let $\P$ be a polytope whose span is contained in the domain of a linear mapping $\phi$.
    Then, $\phi(\P)$ is the convex hull of the images of the vertices of $\P$.
    This further implies that the vertices of $\phi(\P)$ are a subset of the images of the vertices of $\P$ under $\phi$.
\end{lemma}
\begin{proof}
    Let $\vv_1,\dots,\vv_n$ be the vertices of $\P$.
    If $\vx\in\P$, then $\vx$ can be written as a convex combination of the vertices of $\P$ as $\vx=\lambda_1\vv_1+\dots+\lambda_n\vv_n$, where $0\le\lambda_1,\dots,\lambda_n\le1$ and $\lambda_1+\dots+\lambda_n=1$.
    Observe that
    \begin{equation*}
        \phi(\vx) = \phi(\lambda_1\vv_1+\dots+\lambda_n\vv_n) = \lambda_1\phi(\vv_1)+\dots+\lambda_n\phi(\vv_n),
    \end{equation*}
    so $\phi(\vx)$ is a convex combination of the images of the vertices of $\P$ under $\phi$.
    Conversely, if $\vx$ is a convex combination of $\phi(\vv_1),\dots,\phi(\vv_n)$, then $\vx$ is the image of some convex combination of $\vv_1,\dots,\vv_n$,
    \begin{equation*}
        \vx=\lambda_1\phi(\vv_1)+\dots+\lambda_n\phi(\vv_n)=\phi(\lambda_1\vv_1+\dots+\lambda_n\vv_n).
    \end{equation*}
    Thus, $\phi(\P)=\conv\{\phi(\vv_1),\dots,\phi(\vv_n)\}$, as desired.
\end{proof}

\begin{lemma}\label{lemma:psi-mapping}
    Let $\sigma\in S_n$ have cycle decomposition $\sigma=\sigma_1\cdots\sigma_m$ and let $\P$ be a $\sigma$-invariant polytope.
    Additionally, define
    \begin{equation*}
        \psi_\sigma(\vx) = \frac{1}{|\sigma|}\sum_{i=1}^{|\sigma|}\sigma^i\cdot\vx = \sum_{k=1}^{m}\frac{\sum_{j\in\sigma_k}x_j}{|\sigma_k|}\ve_{\sigma_k}.
    \end{equation*}
    Then, $\P^\sigma=\psi_\sigma(\P)$.
\end{lemma}
\begin{proof}
    First, observe that $\sigma\cdot\ve_{\sigma_i}=\ve_{\sigma_i}$ for each $\sigma_i$ in the cycle decomposition of $\sigma$.
    Since $\psi_\sigma(\vx)$ can be expressed as a linear combination of $\ve_{\sigma_1},\dots,\ve_{\sigma_m}$, this implies $\sigma\cdot\psi_\sigma(\vx)=\psi_\sigma(\vx)$ for all $\vx\in\R^n$.
    Since $\P$ is $S_n$-invariant, we also know that $\sigma\cdot\vx\in\P$ for all $\vx\in\P$, and thus $\psi_\sigma(\P)\subseteq\P$.
    Therefore, $\psi_\sigma(\P)\subseteq\P^\sigma$, since $\phi_{\sigma}(\P)$ is a subset of $\P$ that is fixed under the action of $\sigma$.
    Conversely, for any $\vx\in\P^\sigma$, note that $\sigma\cdot\vx=\vx$, and thus $\sigma^i\cdot\vx=\vx$ for all integer powers $i$.
    This implies
    \begin{equation*}
        \psi_\sigma(\vx) = \frac{1}{|\sigma|}\sum_{i=1}^{|\sigma|}\sigma^i\cdot\vx
        = \frac{1}{|\sigma|}\sum_{i=1}^{|\sigma|}\vx
        = \frac{|\sigma|}{|\sigma|}\vx=\vx,
    \end{equation*}
    so $\P^\sigma\subseteq\psi_\sigma(\P)$.
    Thus, we have $\P^\sigma=\psi_\sigma(\P)$.
\end{proof}

Lemma \ref{lemma:mapped-vertices}  and Lemma \ref{lemma:psi-mapping} together allow us to obtain a generating set for the convex hull of polytopes fixed by coordinate permutations.
We can then identify the vertices of each fixed polytope from this set, providing an explicit vertex description, which we record in Theorem \ref{thm:contraction-vert-desc}.

\begin{theorem}\label{thm:contraction-vert-desc}
    Let $G=([n],E)$ be a graph on $n$ vertices and $\sigma\in S_n$ have cycle decomposition $\sigma=\sigma_1\cdots\sigma_m$.
    Then, $\SEP(G)^\sigma$ has vertex set
    \begin{equation*}
        \set{\frac{1}{|\sigma_i|}\ve_{\sigma_i}-\frac{1}{|\sigma_j|}\ve_{\sigma_j}}{(a,b)\in E\text{ for some }a\in\sigma_i\text{ and }b\in\sigma_j}.
    \end{equation*}
\end{theorem}
\begin{proof}
    Define $\psi_\sigma$ as in Lemma \ref{lemma:psi-mapping}.
    For each vertex $\ve_i-\ve_j$ of $\SEP(G)$, note that $\psi_\sigma(\ve_i-\ve_j)=\frac{1}{|\sigma_I|}\ve_{\sigma_I}-\frac{1}{|\sigma_J|}\ve_{\sigma_J}$, where $i\in\sigma_I$ and $j\in\sigma_J$, which is nonzero if and only if $I\ne J$.
    We show that the nonzero points of this form are exactly the vertices of $\SEP(G)^\sigma$.

    First, combining Lemmas \ref{lemma:mapped-vertices} and \ref{lemma:psi-mapping}  ensures
    \begin{equation*}
        \P^\sigma = \psi_\sigma(\P)
        = \conv\set{\phi(\ve_i-\ve_j)}{(i,j)\in E}.
    \end{equation*}
    Using the form of $\phi(\ve_i-\ve_j)$ noted above, and omitting the redundant case when $\phi(\ve_i-\ve_j)=\bv0$, we can see that
    \begin{equation*}
        \P^\sigma=\conv\set{\frac{1}{|\sigma_i|}\ve_{\sigma_i}-\frac{1}{|\sigma_j|}\ve_{\sigma_j}}{(\sigma_i,\sigma_j)\in E'}.
    \end{equation*}
    It remains only to show that each $\frac{1}{|\sigma_i|}\ve_{\sigma_i}-\frac{1}{|\sigma_j|}\ve_{\sigma_j}$ is a vertex of their convex hull.
    Note that each element $\frac{1}{|\sigma_i|}\ve_{\sigma_i}-\frac{1}{|\sigma_j|}\ve_{\sigma_j}$ is the unique maximizer among these generators of the linear functional $\vx\mapsto\left\langle\ve_{\sigma_i}-\ve_{\sigma_j},\vx\right\rangle$, and is thus a vertex of the convex hull.
    Hence, the vertices of $\SEP(G)^\sigma=\psi_\sigma(\SEP(G))$ are exactly $\set{\frac{1}{|\sigma_i|}\ve_{\sigma_i}-\frac{1}{|\sigma_j|}\ve_{\sigma_j}}{(\sigma_i,\sigma_j)\in E'}$.
\end{proof}

From this vertex description, we obtain the equivalence of the fixed polytope to the symmetric edge polytope of a smaller graph.

\begin{corollary}\label{cor:contraction-equiv}
    Let $G=([n],E)$ be a graph on $n$ vertices and $\sigma\in S_n$ have cycle decomposition $\sigma=\sigma_1\cdots\sigma_m$.
    Additionally, viewing $\sigma$ as a partition of $[n]$, let $G'=(\sigma,E')$ be the contraction of $G$ induced by $\sigma$.
    Then, $\SEP(G)^\sigma$ is linearly and combinatorially equivalent to $\SEP(G')$.
\end{corollary}
\begin{proof}
    Note that $\{\ve_{\sigma_i}\}_{i\in[m]}$ is an orthogonal basis of $(\R^n)^\sigma$.
    Then, define the linear transformation $\phi:(\R^n)^\sigma\to\R^m$ such that $\phi(\ve_{\sigma_i})=|\sigma_i|\ve_i$.
    Observe that $\phi$ is bijective since each orthogonal basis vector $\ve_{\sigma_i}$ of $(\R^n)^\sigma$ is mapped to a nonzero scale of a distinct orthogonal basis vector $\ve_i$ of $\R^m$.
    Additionally, $\phi$ maps each vertex $\frac{1}{|\sigma_i|}\ve_{\sigma_i}-\frac{1}{|\sigma_j|}\ve_{\sigma_j}$ of $\SEP(G)^\sigma$ to a distinct vertex $\ve_i-\ve_j$ of $\SEP(G')$.
    Thus, $\SEP(G)^\sigma$ is linearly equivalent to $\SEP(G')$, and this linear equivalence implies combinatorial equivalence as well.
\end{proof}

We can further refine the equivalence between the fixed polytope of the symmetric edge polytope of a graph fixed by a permutation and the symmetric edge polytope of the contraction of that graph induced by the same permutation by demonstrating a relationship between their relative volumes.

\begin{theorem}\label{thm:sep-sub-vol}
    Let $G=([n],E)$ be a graph on $n$ vertices, let $\sigma\in S_n$ have cycle decomposition $\sigma=\sigma_1\cdots\sigma_m$, and let $G'=(\sigma,E')$ be the contraction of $G$ induced by $\sigma$ as a partition of $[n]$.
    Then,
    \begin{equation*}
        \rvol(\SEP(G)^\sigma) = \frac{\gcd(|\sigma_1|,\dots,|\sigma_m|)}{\prod_{i=1}^m|\sigma_i|}\rvol(\SEP(G')).
    \end{equation*}
\end{theorem}

\begin{proof}
    Observe that $\ve_{\sigma_1},\dots,\ve_{\sigma_m}$ form an orthogonal basis of $(\R^n)^\sigma$ and let $\phi:(\R^n)^\sigma\to\R^m$ be the bijective linear map defined by $\phi(\ve_{\sigma_i})=|\sigma_i|\ve_i$.
    Recall that $\phi$ maps $\SEP(G)^\sigma$ to $\SEP(G')$ as demonstrated in Corollary \ref{cor:contraction-equiv}.
    Consider the half-open integral parallelepiped, $\P$, formed by the vectors
    \begin{equation*}
        |\sigma_i||\sigma_m|\left(\frac{1}{|\sigma_i|}\ve_{\sigma_i}-\frac{1}{|\sigma_m|}\ve_{\sigma_m}\right) = |\sigma_m|\ve_{\sigma_i}-|\sigma_i|\ve_{\sigma_m},
    \end{equation*}
    for $i\in[m-1]$.
    That is, $\P=\set{\sum_{i=1}^m\lambda_i(|\sigma_m|\ve_i-|\sigma_i|\ve_m)}{0\le\lambda_1,\dots,\lambda_{m-1}<1,\ \sum_{i=1}^m\lambda_i=1}$.
    Note that $\phi$ maps $\P$ to the integral parallelepiped $\phi(\P)$ formed by the vectors $|\sigma_i||\sigma_m|(\ve_i-\ve_m)$ for $i\in[m-1]$.
    We claim that
    \begin{equation*}
        \frac{\rvol(\P)}{\rvol(\phi(\P))} = \frac{\gcd(|\sigma_1|,\dots,|\sigma_m|)}{\prod_{i\in[m]}|\sigma_i|},
    \end{equation*}
    and thus, since $\Span\P=\Span \SEP(G)^\sigma$ by construction,
    \begin{equation*}
        \frac{\rvol(\SEP(G)^\sigma)}{\rvol(\SEP(G'))} = \frac{\rvol(\SEP(G)^\sigma)}{\rvol(\phi(\SEP(G)^\sigma))} = \frac{\rvol(\P)}{\rvol(\phi(\P))} = \frac{\gcd(|\sigma_1|,\dots,|\sigma_m|)}{\prod_{i\in[m]}|\sigma_i|}.
    \end{equation*}

    Since $\phi(\P)$ is a half-open integral parallelepiped, its relative volume is given by the number of lattice points it contains.
    We can write $\phi(\P)$ more explicitly in terms of the orthogonal basis vectors $\ve_i$ as
    \begin{equation*}
        \phi(\P) = \set{\sum_{i\in[m-1]}|\sigma_i||\sigma_m|\lambda_i\ve_i-\left(\sum_{i\in[m-1]}|\sigma_i||\sigma_m|\lambda_i\right)\ve_m}{0\le\lambda_i<1}.
    \end{equation*}
    Restricting each coordinate to be integral, we see the number of lattice points contained in $\phi(\P)$ is the number of vectors $\lambda=(\lambda_1,\dots,\lambda_{m-1})\in[0,1)^{m-1}$ satisfying $|\sigma_i||\sigma_m|\lambda_i\in\Z$ and $\sum_{i\in[m-1]}|\sigma_i||\sigma_m|\lambda_i\in\Z$.
    Substituting $n_i=|\sigma_i||\sigma_m|\lambda_i$, this is equivalent to the number of integral vectors $\vn\in\Z^{m-1}$ satisfying $0\le n_i<|\sigma_i||\sigma_m|$ and $\sum_{i\in[m-1]}n_i\in\Z$.
    That $\sum_{i\in[m-1]}n_i\in\Z$ follows immediately from the condition that $n_i\in\Z$ for each $i\in[m-1]$, so we need only count the number of integer vectors $\vn$ for which $0\le n_i<|\sigma_i||\sigma_m|$, of which there are
    \begin{equation*}
        \prod_{i\in[m-1]}|\sigma_i||\sigma_m| = |\sigma_m|^{m-2}\prod_{i\in[m]}|\sigma_i|.
    \end{equation*}
    Thus, $\rvol(\phi(\P))=|\sigma_m|^{m-2}\prod_{i\in[m]}|\sigma_i|$.

    We can find the relative volume of $\P$ similarly, by counting the number of lattice points it contains, since $\P$ is also an integral parallelepiped.
    Writing $\P$ in terms of the orthogonal basis vectors $\ve_{\sigma_i}$, we see
    \begin{equation*}
        \P = \set{\sum_{i\in[m-1]}|\sigma_m|\lambda_i\ve_{\sigma_i}-\left(\sum_{i\in[m-1]}|\sigma_i|\lambda_i\right)\ve_{\sigma_m}}{0\le\lambda_i<1}.
    \end{equation*}
    Restricting each coordinate to be integral, we see the number of lattice points contained in $\P$ is the number of vectors $\lambda=(\lambda_1,\dots,\lambda_{m-1})\in[0,1)^{m-1}$ satisfying $|\sigma_m|\lambda_i\in\Z$ and $\sum_{i\in[m-1]}|\sigma_i|\lambda_i\in\Z$.
    Substituting $n_i=|\sigma_m|\lambda_i$, this is equivalent to the number of integral vectors $\vn\in[|\sigma_m|]^{m-1}\subset\Z^{m-1}$ for which $\sum_{i\in[m-1]}|\sigma_i|n_i\in|\sigma_m|\Z$.

    Let $g=\gcd(|\sigma_1|,\dots,|\sigma_m|)$.
    Then, $\sum_{i\in[m-1]}|\sigma_i|n_i\equiv0\pmod{g}$ for all $\vn\in\Z^{m-1}$ and there exists some $\vc\in[|\sigma_m|]^{m-1}$ for which $\sum_{i\in[m-1]}|\sigma_i|c_i\equiv g\pmod{|\sigma_m|}$.
    Now, for any $k\in\left[\frac{|\sigma_m|}{g}\right]$, define
    \begin{equation*}
        \psi_{k}:[|\sigma_m|]^{m-1}\to[|\sigma_m|]^{m-1} \;\text{ as }\; \psi_{k}(\vn)\equiv\vn-k\vc\pmod{|\sigma_m|},
    \end{equation*}
    where the reduction mod $|\sigma_m|$ is performed element-wise.
    We claim $\psi_k$ provides a bijection from vectors $\vn\in[|\sigma_m|]^{m-1}$ for which $\sum_{i\in[m-1]}|\sigma_i|n_i\equiv kg\pmod{|\sigma_m|}$ to vectors $\vn'\in[|\sigma_m|]^{m-1}$ for which $\sum_{i\in[m-1]}|\sigma_i|n_i'\equiv0\pmod{|\sigma_m|}$.
    If $\sum_{i\in[m-1]}|\sigma_i|n_i=kg\pmod{|\sigma_m|}$, then \begin{equation*}
        \sum_{i\in[m-1]}|\sigma_i|\psi_k(\vn)_i
        \equiv \sum_{i\in[m-1]}|\sigma_i|(n_i+kc_i)
        = \sum_{i\in[m-1]}|\sigma_i|n_i+k\sum_{i\in[m-1]}|\sigma_i|c_i
        \equiv kg-kg\equiv0
        \pmod{|\sigma_m|},
    \end{equation*}
    as desired.
    Since such a bijection can be found for each possible sum $0g,1g,\dots,\left(\frac{|\sigma_m|}{g}-1\right)g$, there must be the same number of vectors $\vn\in[|\sigma_m|]^{m-1}$ producing sums congruent mod $|\sigma_m|$ to each of these $g$ distinct values.
    Thus, there are exactly $\frac{|\sigma_m|^{m-1}}{\frac{|\sigma_m|}{g}}=|\sigma_m|^{m-2}g$ vectors $\vn\in[|\sigma_m|]^{m-1}$ for which $\sum_{i\in[m-1]}|\sigma_i|n_i\equiv0\pmod{|\sigma_m|}$, so $\rvol(\P)=|\sigma_m|^{m-2}\gcd(|\sigma_1|,\dots,|\sigma_m|)$.

    Together with the value of $\rvol(\phi(\P))$ obtained above, we have
    \begin{equation*}
        \frac{\rvol(\P)}{\rvol(\phi(\P))} = \frac{|\sigma_m|^{m-2}\gcd(|\sigma_1|,\dots,|\sigma_m|)}{|\sigma_m|^{m-2}\prod_{i\in[m]}|\sigma_i|}
        = \frac{\gcd(|\sigma_1|,\dots,|\sigma_m|)}{\prod_{i\in[m]}|\sigma_i|},
    \end{equation*}
    as desired.
    This implies
    \begin{equation*}
        \frac{\rvol(\SEP(G)^\sigma)}{\rvol(\SEP(G'))} = \frac{\rvol(\SEP(G)^\sigma)}{\rvol(\phi(\SEP(G)^\sigma))}
        = \frac{\rvol(\P)}{\rvol(\phi(\P))}
        = \frac{\gcd(|\sigma_1|,\dots,|\sigma_m|)}{\prod_{i\in[m]}|\sigma_i|},
    \end{equation*}
    and so
    \begin{equation*}
        \rvol(\SEP(G)^\sigma) = \frac{\gcd(|\sigma_1|,\dots,|\sigma_m|)}{\prod_{i=1}^m|\sigma_i|}\rvol(\SEP(G')),
    \end{equation*}
    concluding the proof.
\end{proof}

\section{Fixed Polytopes of \texorpdfstring{$\SEP(K_n)$}{SEP(Kn)}} \label{sec:fixed_root_polytopes}

In this section, we specialize the results of Section \ref{sec:fixed_polytopes} to the special case of complete graphs, $K_n$, which yield a family of polytopes of arising from structures in diverse branches of mathematics.
In particular, we obtain explicit volume formulas and inequality descriptions for the fixed polytopes of $\SEP(K_n)$.

The symmetric edge polytopes of complete graphs, $\SEP(K_n)$ have been studied from a variety of perspectives, under a multitude of names.
For example, $\SEP(K_n)$ is exactly the root polytope of type $A_{n}$ introduced by \cite{GelfandGraevPostnikov} and studied further by \cite{Postnikov}.
Further, Hetyei refers to the same polytope as the Legendre polytope, $\mathcal{L}_{n-1}$, with connections to Legendre polynomials and Delannoy numbers in \cite{Hetyei}.
There, Hetyei notes that these polytopes are the intersection of the cross polytope in dimension $n$ with the hyperplane normal to the all ones vector.
Meanwhile, $\SEP(K_n)$ is the convex hull of the set of points whose coordinates are permutations of the $n$-element tuple $(1,0,\dots,0,-1)$, and are thus generalized permutahedra, $\Pi_n(1,0,\dots,0,-1)$ \cite{Postnikov,pos-root}.

First, we provide a vertex description of $\SEP(K_{n})$ which follows as a corollary of Theorem \ref{thm:contraction-vert-desc}.

\begin{corollary}\label{thm:contraction-vert-desc-complete-graph}
    Let $\sigma\in S_n$ have cycle decomposition $\sigma=\sigma_1\cdots\sigma_m$.
    Then, $\SEP(K_{n})^\sigma$ has vertex set \[\set{\frac{1}{|\sigma_i|}\ve_{\sigma_i}-\frac{1}{|\sigma_j|}\ve_{\sigma_j}}{i,j\in[m],i\ne j}.\]
\end{corollary}
\begin{proof}
    Consider $G=K_{n}$.
    By Theorem \ref{thm:contraction-vert-desc}, the vertices of $\SEP(K_{n})^{\sigma}=\psi_{\sigma}(\SEP(K_{n}))$ are exactly $\set{\frac{1}{|\sigma_i|}\ve_{\sigma_i}-\frac{1}{|\sigma_j|}\ve_{\sigma_j}}{(a,b)\in E(K_{n})\text{ for some }a\in\sigma_i\text{ and }b\in\sigma_j}$, but $K_n$ contains edges between all pairs of distinct vertices, so the membership condition of this set is true for all $i,j\in[m]$ with $i\ne j$.
\end{proof}

We now identify combinatorial and linear equivalence of the fixed polytope $SEP(K_{n})^{\sigma}$ to a symmetric edge polytope determined by the length of a permutation.
More specifically, for $\sigma\in S_{n}$ with cycle decomposition $\sigma=\sigma_{1}\cdots\sigma_{m}$, we identify combinatorial and linear equivalence between $SEP(K_{n})^{\sigma}$ and $SEP(K_{m})$.
We then provide insight into the dimension of the fixed polytope as well as its $H$ description.

\begin{corollary}\label{corollary: SEP-comb-equivalence-lower-dim}
Let $\sigma \in S_{n}$ have cycle decomposition $\sigma=\sigma_1\cdots\sigma_m$.
Then, $\SEP(K_n)^\sigma$ is combinatorially and linearly equivalent to $\SEP(K_m)$.
\end{corollary}
\begin{proof}
    Note that $K_{m}$ corresponds to the contraction of $K_{n}$ induced by $\sigma$.
    Hence, by Corollary \ref{cor:contraction-equiv}, $\SEP(K_{n})^{\sigma}$ is combinatorially and linearly equivalent to $\SEP(K_{m})$.
\end{proof}

\begin{corollary}\label{corollary: SEP-dimension-lower-dim}
Let $\sigma \in S_{n}$ have cycle decomposition $\sigma=\sigma_1\cdots\sigma_m$.
Then $\SEP(K_{n})^{\sigma}$ has dimension $m-1$.
\end{corollary}
\begin{proof}
    By Corollary \ref{corollary: SEP-comb-equivalence-lower-dim}, we have that $\SEP(K_{n})^{\sigma}$ is combinatorially and linearly equivalent to $\SEP(K_{m})$, which has dimension $m-1$, so $\dim\SEP(K_n)^\sigma=\dim\SEP(K_m)=m-1$.
\end{proof}
The results above demonstrate how a fixed polytope $SEP(K_n)^\sigma$ captures a lower-dimensional polytope $SEP(K_m)$, where $m$ is determined by the number of cycles in $\sigma$.
The following example concretely shows how a nontrivial permutation reduces ambient dimension and yields a fixed polytope linearly and combinatorially equivalent to $SEP(K_m)$.

\begin{example}
Consider the symmetric edge polytope associated with $K_{4}$ in Figure \ref{fig:k_3-vs-k_4}.
Notice that $\SEP(K_{4})^{(12)}$ is linearly and combinatorially equivalent to $\SEP(K_{3})$ through a bijection of their face lattices.
In addition, we have $\dim\SEP(K_{4})^{(12)}=\dim\SEP(K_{3})=2$.
\end{example}

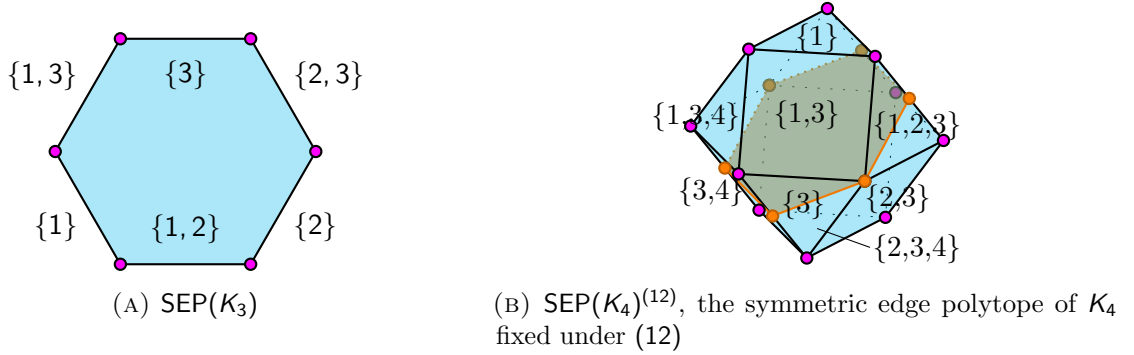
\begin{figure*}[ht]
    \begin{subfigure}{0.5\textwidth}
        \centering
        \begin{tikzpicture}[
            x={(-0.866cm,-0.5cm)},
            y={(0.866cm,-0.5cm)},
            z={(0cm,1cm)},
            scale = 1.0
            ]
            \coordinate (ab) at (1,-1,0);
            \coordinate (ac) at (1,0,-1);
            \coordinate (ba) at (-1,1,0);
            \coordinate (bc) at (0,1,-1);
            \coordinate (ca) at (-1,0,1);
            \coordinate (cb) at (0,-1,1);
            \fill[facet] (ab) -- (ac) -- (bc) -- (ba) -- (ca) -- (cb) -- cycle {};
            \draw[edge] (ab) -- (ac) node[pos=0.67, label={[black]left:$\{1\}$}] {};
            \draw[edge] (ac) -- (bc) node[pos=0.5, label={[black]above:$\{1,2\}$}] {};
            \draw[edge] (bc) -- (ba) node[pos=0.33, label={[black]right:$\{2\}$}] {};
            \draw[edge] (ba) -- (ca) node[pos=0.67, label={[black]right:$\{2,3\}$}] {};
            \draw[edge] (ca) -- (cb) node[pos=0.5, label={[black]below:$\{3\}$}] {};
            \draw[edge] (cb) -- (ab) node[pos=0.33, label={[black]left:$\{1,3\}$}] {};
            \node[vertex] at (ab) {};
            \node[vertex] at (ac) {};
            \node[vertex] at (ba) {};
            \node[vertex] at (bc) {};
            \node[vertex] at (ca) {};
            \node[vertex] at (cb) {};
        \end{tikzpicture}
        \caption{$\SEP(K_3)$\\\vphantom{.}}
    \end{subfigure}%
    \begin{subfigure}{0.5\textwidth}
        \centering
        \begin{tikzpicture}[
            x={(0.374024cm, -0.375379cm)},
            y={(0.927419cm, 0.151412cm)},
            z={(-0.000023cm, 0.914421cm)},
            scale=1.2,
            ]
            \draw[edge,back] (0.00000, 0.81650, 2.30940) -- (0.00000, 1.63299, 1.15470);
            \draw[edge,back] (0.00000, 0.81650, 2.30940) -- (-0.70711, 0.40825, 1.15470);
            \draw[edge,back] (1.41421, 1.63299, 1.15470) -- (0.00000, 1.63299, 1.15470);
            \draw[edge,back] (0.00000, 1.63299, 1.15470) -- (-0.70711, 0.40825, 1.15470);
            \draw[edge,back] (0.00000, 1.63299, 1.15470) -- (0.70711, 1.22474, 0.00000);
            \draw[edge,back] (-0.70711, 0.40825, 1.15470) -- (0.00000, -0.81650, 1.15470);
            \draw[edge,back] (-0.70711, 0.40825, 1.15470) -- (0.00000, 0.00000, 0.00000);
            \draw[edge,back] (0.70711, 1.22474, 0.00000) -- (0.00000, 0.00000, 0.00000);
                \draw[subedge,subback] (0.00000, 1.22474, 1.73205) -- (1.41421, 1.22474, 1.73205);
                \draw[subedge,subback] (0.00000, 1.22474, 1.73205) -- (-0.70711, 0.40825, 1.15470);
                \draw[subedge,subback] (-0.70711, 0.40825, 1.15470) -- (0.00000, -0.40825, 0.57735);
            \node[vertex] at (0.00000, 1.63299, 1.15470)     {};
            \node[vertex] at (-0.70711, 0.40825, 1.15470)     {};
                \fill[subfacet] (0.00000, -0.40825, 0.57735) -- (-0.70711, 0.40825, 1.15470) -- (0.00000, 1.22474, 1.73205) -- (1.41421, 1.22474, 1.73205) -- (2.12132, 0.40825, 1.15470) -- (1.41421, -0.40825, 0.57735) -- cycle {};
                \node[subvertex] at (0.00000, 1.22474, 1.73205)     {};
                \node[subvertex] at (-0.70711, 0.40825, 1.15470)     {};
            \fill[facet] (0.70711, -0.40825, 2.30940) -- (1.41421, 0.81650, 2.30940) -- (0.00000, 0.81650, 2.30940) -- cycle {};
            \fill[facet] (2.12132, 0.40825, 1.15470) -- (1.41421, 0.81650, 2.30940) -- (1.41421, 1.63299, 1.15470) -- cycle {};
            \fill[facet] (1.41421, -0.81650, 1.15470) -- (0.70711, -0.40825, 2.30940) -- (1.41421, 0.81650, 2.30940) -- (2.12132, 0.40825, 1.15470) -- cycle {};
            \fill[facet] (0.00000, -0.81650, 1.15470) -- (0.70711, -0.40825, 2.30940) -- (1.41421, -0.81650, 1.15470) -- cycle {};
            \fill[facet] (1.41421, 0.00000, 0.00000) -- (2.12132, 0.40825, 1.15470) -- (1.41421, 1.63299, 1.15470) -- (0.70711, 1.22474, 0.00000) -- cycle {};
            \fill[facet] (1.41421, 0.00000, 0.00000) -- (2.12132, 0.40825, 1.15470) -- (1.41421, -0.81650, 1.15470) -- cycle {};
            \fill[facet] (0.00000, 0.00000, 0.00000) -- (0.00000, -0.81650, 1.15470) -- (1.41421, -0.81650, 1.15470) -- (1.41421, 0.00000, 0.00000) -- cycle {};
            \draw[edge] (1.41421, 0.81650, 2.30940) -- (0.00000, 0.81650, 2.30940);
            \draw[edge] (1.41421, 0.81650, 2.30940) -- (0.70711, -0.40825, 2.30940);
            \draw[edge] (1.41421, 0.81650, 2.30940) -- (1.41421, 1.63299, 1.15470);
            \draw[edge] (1.41421, 0.81650, 2.30940) -- (2.12132, 0.40825, 1.15470);
            \draw[edge] (0.00000, 0.81650, 2.30940) -- (0.70711, -0.40825, 2.30940);
            \draw[edge] (0.70711, -0.40825, 2.30940) -- (1.41421, -0.81650, 1.15470);
            \draw[edge] (0.70711, -0.40825, 2.30940) -- (0.00000, -0.81650, 1.15470);
            \draw[edge] (1.41421, 1.63299, 1.15470) -- (2.12132, 0.40825, 1.15470);
            \draw[edge] (1.41421, 1.63299, 1.15470) -- (0.70711, 1.22474, 0.00000);
            \draw[edge] (2.12132, 0.40825, 1.15470) -- (1.41421, -0.81650, 1.15470);
            \draw[edge] (2.12132, 0.40825, 1.15470) -- (1.41421, 0.00000, 0.00000);
            \draw[edge] (1.41421, -0.81650, 1.15470) -- (0.00000, -0.81650, 1.15470);
            \draw[edge] (1.41421, -0.81650, 1.15470) -- (1.41421, 0.00000, 0.00000);
            \draw[edge] (0.00000, -0.81650, 1.15470) -- (0.00000, 0.00000, 0.00000);
            \draw[edge] (0.70711, 1.22474, 0.00000) -- (1.41421, 0.00000, 0.00000);
            \draw[edge] (1.41421, 0.00000, 0.00000) -- (0.00000, 0.00000, 0.00000);
                \draw[subedge] (1.41421, 1.22474, 1.73205) -- (2.12132, 0.40825, 1.15470);
                \draw[subedge] (2.12132, 0.40825, 1.15470) -- (1.41421, -0.40825, 0.57735);
                \draw[subedge] (1.41421, -0.40825, 0.57735) -- (0.00000, -0.40825, 0.57735);
            \node[vertex] at (1.41421, 0.81650, 2.30940)     {};
            \node[vertex] at (0.00000, 0.81650, 2.30940)     {};
            \node[vertex] at (0.70711, -0.40825, 2.30940)     {};
            \node[vertex] at (1.41421, 1.63299, 1.15470)     {};
            \node[vertex] at (2.12132, 0.40825, 1.15470)     {};
            \node[vertex] at (1.41421, -0.81650, 1.15470)     {};
            \node[vertex] at (0.00000, -0.81650, 1.15470)     {};
            \node[vertex] at (0.70711, 1.22474, 0.00000)     {};
            \node[vertex] at (1.41421, 0.00000, 0.00000)     {};
            \node[vertex] at (0.00000, 0.00000, 0.00000)     {};
                \node[subvertex] at (1.41421, 1.22474, 1.73205)     {};
                \node[subvertex] at (2.12132, 0.40825, 1.15470)     {};
                \node[subvertex] at (1.41421, -0.40825, 0.57735)     {};
                \node[subvertex] at (0.00000, -0.40825, 0.57735)     {};
            \node at (0.70711, 0.40825, 2.30940) {\{1\}};
            \node[xshift=0.3cm] at (1.64991, 0.95258, 1.53960) {\{1,2,3\}};
            \node at (1.41421, 0.00000, 1.73205) {\{1,3\}};
            \node[xshift=-0.4cm] at (0.70711, -0.68042, 1.53960) {\{1,3,4\}};
            \node[xshift=0.3cm] at (1.41421, 0.81650, 0.57735) {\{2,3\}};
            \node[xshift=0.0cm] at (1.64991, -0.13608, 0.76980) {\{3\}};
            \node[xshift=-0.5cm] at (0.70711, -0.40825, 0.57735) {\{3,4\}};
            \draw[black, thin] (0.70711, 0.40825, 0.00000) -- (1.40711, 0.75825, 0.00000) node[label={[right,yshift=-0.15cm,xshift=-0.1cm]:\{2,3,4\}}] {};
        \end{tikzpicture}
        \caption{$\SEP(K_4)^{(12)}$, the symmetric edge polytope of $K_4$ fixed under $(12)$}
    \end{subfigure}
    \caption{Correspondence between $SEP(K_3)$ and $SEP(K_4)^{(1,2)}$}
    \label{fig:k_3-vs-k_4}
\end{figure*}

We now compute the relative volume of the fixed subpolytopes of symmetric edge polytopes of complete graphs under arbitrary coordinate permutations, leveraging the volume relationship obtained in Theorem \ref{thm:sep-sub-vol} and an explicit volume formula for $\SEP(K_n)$ as found in \cite{siam-root}.
\begin{corollary}\label{cor:sep-km-vol}
    Let $\sigma\in S_n$ have cycle decomposition $\sigma=\sigma_1\cdots\sigma_m$.
    Then, \[\rvol(\SEP(K_n)^\sigma)=\frac{\gcd(|\sigma_1|,\dots,|\sigma_m|)}{(m-1)!\prod_{i=1}^m|\sigma_i|}\binom{2(m-1)}{m-1}.\]
\end{corollary}
\begin{proof}
    From \cite{siam-root} we have $\rvol(\SEP(K_m))=\frac{1}{(m-1)!}\binom{2(m-1)}{m-1}$.
    Then, applying Theorem \ref{thm:sep-sub-vol}, we see \[\rvol(\SEP(K_n)^\sigma)=\frac{\gcd(|\sigma_1|,\dots,|\sigma_m|)}{\prod_{i=1}^m|\sigma_i|}\rvol(\SEP(K_m))=\frac{\gcd(|\sigma_1|,\dots,|\sigma_m|)}{(m-1)!\prod_{i=1}^m|\sigma_i|}\binom{2(m-1)}{m-1},\] as desired.
\end{proof}

We provide an alternative inequality description for the symmetric edge polytope of the complete graph to facilitate later analysis of facet structure.
\begin{corollary}\label{cor:h-desc}
The symmetric edge polytope $\SEP(K_n)$ is the set of points $\vx\in\R^n$ such that ${\bv1}^\top\vx=0$ and $\left\lvert\ve_S^\top\vx\right\rvert\le1$ for all nonempty proper subsets $S\subset[n]$ with $|S|\le\frac{n}{2}$.
    Further, the half-space $\ve_S^\top\vx\ge-1$ in this description is equivalent to the half-space $\ve_{S^C}^\top\vx\le1$ in Proposition \ref{prop:11} when restricted to ${\bv1}^\top\vx=0$.
\end{corollary}
\begin{proof}
    Let $\vx\in\R^n$ such that ${\bv1}^\top\vx=0$ and $S$ be any nonempty proper subset of $[n]$.
    We will show that $\ve_S^\top\vx\ge-1$ if and only if $\ve_{S^C}^\top\vx\le1$; the desired inequality description and the noted correspondence with Proposition \ref{prop:11} then follow.
    Observe that $\ve_{S^C}=\bv1-\ve_S$, so \[\ve_{S^C}^\top\vx=(\bv1-\ve_S)^\top\vx=-\ve_S^\top\vx.\]
    Thus, $\ve_S^\top\vx\ge-1$ if and only if $-\ve_S^\top\vx\le1$ if and only if $\ve_{S^C}^\top\vx\le1$, as desired.
\end{proof}
The corollary above captures the centrally symmetric structure of $SEP(K_n)$.
By restricting the inequalities to an invariant subspace, we are able to characterize facets of the fixed polytope induced by the original facets and further count the facets of $SEP(K_n)^\sigma$.

\begin{theorem}\label{thm:SEP-H-des-cycle-decomp}
    Let $\sigma \in S_{n}$ have cycle decomposition $\sigma_{1}\cdots\sigma_{m}$, and label the facets of $\SEP(K_{n})$ each with proper subsets of $[n]$ as in Proposition \ref{prop:11}.
    Then $\SEP(K_{n})^{\sigma}$ has exactly $2^{m}-2$ facets.
    Each facet is contained in a distinct facet of $\SEP(K_{n})$ whose labels satisfy $S \cap \sigma_{i} = \emptyset$ or $S \cap \sigma_{k}=\sigma_{k}$ for each $k \in [m]$.
\end{theorem}
\begin{proof}
    First, we will show that the interior of each facet of $\SEP(K_n)$ whose label contains all of $\sigma_k$ or none of $\sigma_k$ for each $k\in[m]$ contains a point in $\SEP(K_n)^\sigma$ and thus contains a facet of $\SEP(K_n)^\sigma$.

    Consider a facet with label $S$ satisfying this condition, and define $\vx=\frac{1}{|S||S^C|}\sum_{(i,j)\in S\times S^C}(\ve_i-\ve_j)$, which is a convex combination of vectors of the form $\ve_i-\ve_j$, and thus contained in $\SEP(K_n)$.
    Rearranging, we can see \[\vx=\frac{1}{|S|}\sum_{i\in S}\ve_i-\frac{1}{|S^C|}\sum_{j\in S^C}\ve_j=\frac{1}{|S|}\ve_S-\frac{1}{|S^C|}\ve_{S^C},\] so $\ve_S^\top\vx=\frac{|S|}{|S|}=1$, meaning $\vx$ is on the facet with label $S$.
    Further, for any $S'\ne S\subset[n]$,
    \begin{align*}
    \ve_{S'}^\top\vx & =(\ve_S+\ve_{S'\setminus S}-\ve_{S\setminus S'})^\top\vx\\ 
    & =\ve_S^\top\vx+\ve_{S'\setminus S}^\top\vx-\ve_{S\setminus S'}^\top\vx.
    \end{align*}
    Noting that $\ve_{S'\setminus S}^\top\vx\le0$,\; $\ve_{S\setminus S'}^\top\vx\ge0$,\; $\ve_S^\top\vx=1$, and that $\ve_{S'\setminus S}^\top\vx$ and $\ve_{S\setminus S}^\top\vx$ cannot both be zero, we have $\ve_{S'}^\top\vx<1$, so $\vx$ is not on any other facet of $\SEP(K_n)$.
    Thus, $\vx$ is in the interior of the facet of $\SEP(K_n)$ with label $S$.

    Additionally, note that whether $\sigma_{k}\subset S$ or $S\cap\sigma_{k}=\emptyset$, $\sigma_{k}\cdot S=S$, so for all $k \in [m]$, 
    \begin{align*}
        \sigma_{k}\cdot\vx &=\frac{1}{|S|}\sum_{i\in S}\ve_{\sigma_{k}(i)}-\frac{1}{|S^C|}\sum_{j\in S^C}\ve_{\sigma_{k}(j)}\\
        &=\frac{1}{|S|}\sum_{i\in\sigma_{k}(S)}\ve_i-\frac{1}{|S^C|}\sum_{i\in\sigma_{k}(S^C)}\ve_j\\
        &=\frac{1}{|S|}\sum_{i\in S}\ve_i-\frac{1}{|S^C|}\sum_{j\in S^C}\ve_j=\vx.
    \end{align*}
    Thus, $\vx\in\SEP(K_n)^\sigma$, so $\SEP(K_n)^\sigma$ contains a point in the interior of the facet of $\SEP(K_n)$ with label $S$, and thus has a facet contained in the facet of $\SEP(K_n)$ with label $S$.
    Note that if $S$ satisfies this condition, then $S^C$ also does, so the inequalities $\ve_S^\top\vx\le1$ and $\ve_{S^C}^\top\vx\le1$ are both included in the inequality description of $\SEP(K_n)^\sigma$.
    These inequalities can be paired to write $|\ve_S^\top\vx|\le1$ as in Corollary \ref{cor:h-desc}.

    Now, consider a facet with label $S$ such that $\emptyset\subsetneq S\cap\sigma_{k}\subsetneq\sigma_{k}$ for some $k\in[m]$.
    Note that $0\le(\ve_S)_i\le(\ve_{S\cup\sigma_{k}})_i$ for all $i\in[n]$.
    Recall that for any $\va,\vb \in \mathbb{R}^{n}$ such that $0 \leq a_{i} \leq b_{i}$ for all $i \in [n]$, we have $|\va^{\top}\vx|\leq |\vb^{\top}\vx|$ for all $\vx \in \mathbb{R}^{n}$.
    Hence, $\left\lvert\ve_{S\cup\sigma_{k}}^\top\vx\right\rvert\le1$ implies $\left\lvert\ve_S^\top\vx\right\rvert\le1$, so the latter pair of inequalities is redundant given the former.
    Since the former inequalities come from facets of $\SEP(K_n)$ with labels that include all of $\sigma_{k}$ or none of it, we have already shown that the former inequalities are included in the inequality description of $\SEP(K_n)^\sigma$.
    Thus, the half-planes defined by facets of $\SEP(K_n)$ with labels for which $\emptyset\subsetneq S\cap\sigma_{k}\subsetneq\sigma_{k}$ are not needed in the inequality description of $\SEP(K_n)^\sigma$ and do not correspond to its facets.
\end{proof}
 
\section*{Acknowledgments} 
We are grateful to the Mathematics Department at Harvey Mudd College for providing a wonderful environment to produce our work.
We are grateful to Michael Orrison and Dagan Karp for feedback on our presentation of results.
We further thank DruAnn Thomas and Jon Jacobson for their support.
TAC and ARVM thank the Rose Hills Foundation for providing partial funding to pursue this research.


\bibliographystyle{amsplain}
\bibliography{equivariant-references}

\end{document}